\documentclass{amsart} 
\usepackage{amssymb, amsmath}
\usepackage{amsfonts}
\usepackage{amssymb}
\usepackage{color}
\usepackage{comment}
\usepackage{epsfig}
\usepackage{float}
\usepackage{graphicx}
\newtheorem{theorem}{Theorem}[section]
\newtheorem*{thmA*}{Theorem A}
\newtheorem*{thmB*}{Theorem B}
\newtheorem*{lemA*}{Lemma A}
\newtheorem{corollary}[theorem]{Corollary}

\newtheorem{definition}{Definition}

\newtheorem{lemma}{Lemma}[section]
\newtheorem{prop}{Proposition}[section]

\newtheorem{proposition}[theorem]{Proposition}
\newtheorem{remark}{Remark}

\newtheorem*{PropertyA*}{Property A}
\newtheorem*{assumptionA*}{Assumption A}
\newtheorem*{assumptionB*}{Assumption B}
\newtheorem*{assumptionC*}{Assumption C}
\newtheorem*{assumptionOmega*}{Assumption $\Omega$ }
\newtheorem{Winfree Model}{Winfree Model}

\newcommand{\R}{\mathbb R}

\newcommand{\diam}{\text{\rm diam}}
\newcommand{\gain}{\text{\rm gain}}
\newcommand{\loss}{\text{\rm loss}}

\def\be{\begin{equation}}
\def\ee{\end{equation}}

\DeclareMathOperator*{\esssup}{ess\,sup}

\def\R{\mathbb R}

\numberwithin{equation}{section} \numberwithin{theorem}{section}
\numberwithin{figure}{section}
\vspace{1cm}

\usepackage{array,makecell}
\usepackage{graphicx, array, blindtext}

\allowdisplaybreaks

\begin{document}
\bibliographystyle{siam}

\title[]
{On the Existence and Regularity for Stationary Boltzmann Equation in a Small Domain}

\author[I-K.~Chen, C.-H.~Hsia, D.~Kawagoe and J.-K.~Su]{I-KUN CHEN, CHUN-HSIUNG HSIA, DAISUKE KAWAGOE AND JHE-KUAN SU}

\date{\today}

\begin{abstract}
In this article, we study the stationary Boltzmann equation with the incoming boundary condition for the hard potential cases. Assuming the smallness of the domain and a suitable normal curvature condition on the boundary, we find a suitable solution space which is a proper subset of the $W^{1,p}$ space for $1\leq p < 3$. 
\end{abstract}

\maketitle

\section{Introduction}

The Boltzmann equation, which reads as
\begin{equation} \label{Boltzmann equation time evolutional}
\frac{\partial F}{\partial t} + v \cdot \nabla_{x} F = Q(F,F), \quad (t, x, v) \in \R \times \Omega \times \R^3 ,
\end{equation}
describes the dynamical behavior of rarefied gas molecular in the space domain $\Omega$ in $\mathbb{R}^3$ with velocity in $\mathbb{R}^3$ at time $t$. The function $Q$ is called the collision operator and it is defined by
\begin{equation} \label{collision operator}
Q(F,G):=\int_{\mathbb{R}^{3}}\int_{0}^{2\pi}\int_{0}^{\frac{\pi}{2}}[F(v')G(v_{*}')-F(v)G(v_{*})]B(\vert v-v_{*} \vert , \theta)d\theta d\phi dv_{*},
\end{equation}
where
\begin{align*}
&v':=v+((v_*-v)\cdot \omega)\omega,\,
v_*':=v_*-((v_*-v)\cdot \omega)\omega, \\
&\omega:=\cos{\theta} \frac{v_*-v}{|v_*-v|} + (\sin{\theta}\cos{\phi}) e_2 + (\sin{\theta}\sin{\phi}) e_3.
\end{align*}
Here, $\phi$ and $\theta$ are real values satisfying $0\leq \phi \leq 2 \pi$ and $\ 0 \leq \theta \leq \pi/2$, and the vectors $e_2$, $e_3$ are chosen so that the pair $\{\frac{v_*-v}{|v_*-v|}, e_2, e_3\}$ forms an orthonormal basis. The cross section $B$ describes the interaction between two gas particles when they collide. 

In this article, we consider the existence and regularity theory of the stationary Boltzmann equation
\begin{equation} \label{Boltzmann equation original}
 v \cdot \nabla_{x} F = Q(F,F), \quad (x, v) \in \Omega \times \R^3
\end{equation}
with the incoming boundary condition
\begin{equation} \label{boundary condition original}
F(x,v)=G(x,v),\quad (x,v)  \in \Gamma^-.
\end{equation}
Here, the outgoing and incoming boundaries are defined $\Gamma^\pm := \{ (x,v)\in \Omega \times \mathbb{R}^3 \mid \pm n(x)\cdot v > 0 \}$ with $n(x)$ the outward unit normal of $\partial \Omega$ at $x$. 

When we consider the fluctuation of the solution to the equation \eqref{Boltzmann equation original} from the Maxwellian potential $M:=\pi^{-\frac{3}{2}}e^{-|v|^2}$, that is, $F=M+M^{\frac{1}{2}}f$ and $G = M + M^{\frac{1}{2}}g$, the problem is reduced to the equation
\begin{equation} \label{Boltzmann equation}
 v \cdot \nabla_{x}f = L(f) + \Gamma(f,f), \quad (x, v) \in \Omega \times \R^3
\end{equation}
with the incoming boundary condition
\begin{equation} \label{boundary condition}
f(x,v)=g(x,v),\quad (x,v)  \in \Gamma^-,
\end{equation}
where the operators $L$ and $\Gamma$ are defined by
\[
L(h):= M^{-\frac{1}{2}}\left(Q(M,M^{\frac{1}{2}}h)+Q(M^{\frac{1}{2}}h,M))\right)
\]
and
\[
\Gamma(h_1,h_2):= M^{-\frac{1}{2}}Q(M^{\frac{1}{2}}h_1,M^{\frac{1}{2}}h_2),
\]
respectively. Instead of discussing the boundary value problem \eqref{Boltzmann equation original}-\eqref{boundary condition original}, we shall discuss the boundary value problem \eqref{Boltzmann equation}-\eqref{boundary condition}.

The existence and regularity issues of \eqref{Boltzmann equation} have attracted the attention of many authors. In the 1970s, Guiraud proved the existence of solutions to the stationary Boltzmann equation for both linear and nonlinear cases in convex domains \cite{Guir 1, Guir 2}. For general domains, Esposito, Guo, Kim and Marra proved the existence result for the nonlinear case with the diffuse reflection boundary condition \cite{Eso 1}. Since then, various results had been developed for different boundary conditions. For instance, with some conditions on the boundary, a $C^1$ time-dependent solution to the Boltzmann equation with external fields is constructed in \cite{Cao 1, Cao 2}. In \cite{Guo 1}, a $W^{1,p}$ time-evolutionary solution for $1<p<2$ and a weighted $W^{1,p}$ solution for $2\leq p \leq \infty$ are constructed. In \cite{Ikun 3}, a point-wise estimate for the first derivatives of a solution to the linearized Boltzmann equation with the diffuse reflection boundary condition has been achieved. The space $H^{1-}$ regularity of solutions for the linear case with the incoming boundary condition was achieved in \cite{Ikun 2}. In \cite{ChenKim 1}, a $ C^{ 1,\beta }$ solution to the stationary Boltzmann equation with the diffuse reflection boundary condition is constructed. Recently, in \cite{ChenKim 2}, the existence of the $W^{1,p}_x$ solution of Boltzmann equation with the diffuse reflection boundary condition for $1\leq p < 3$ is proved, which only established the space regularity. 

Nevertheless, for the linearized Boltzmann equation, the authors established the $H^1$ regularity, for both space and velocity variables, for a small domain with the positive Gaussian curvature \cite{Ikun 1}. In \cite{Ikun 1'}, the effect of the geometry on the regularity is investigated. The authors prove the existence of solutions in $W^{1,p}$ spaces for $1\leq p<2$. In contrast, if the positivity of the Gaussian curvature on the boundary is imposed, the authors prove the existence of solutions in $W^{1, p}$ spaces for $1 \leq p < 3$.  In both cases, counterexamples are also provided. 

In this article, with more conditions like uniform sphere conditions for the space domain and the smallness of the incoming data, we derive a point-wise estimate over the gradient of the solution of the nonlinear Boltzmann equation, by which we can derive a $W^{1,p}$ solution for $1 \leq p <3$.

To study the existence of a solution to the Boltzmann equation, we regard it as the linearized Boltzmann equation with a source term:
\begin{equation*} 
 v \cdot \nabla_{x}f = L(f)+\phi
\end{equation*}
with $\phi=\Gamma(f,f)$. Then we apply the following iteration scheme 
\begin{equation} \label{Nonlinear Boltmann iteration scheme}
 v \cdot \nabla_{x}f_{i+1} = L(f_{i+1}) + \Gamma(f_{i},f_{i})
\end{equation}
to prove that the sequence $\{f_{i}\}$ converges to a solution of the nonlinear Boltzmann equation.

Next, we state the detailed setup of our problem. In this article, we consider the following cross section $B$ in \eqref{collision operator}: 
\begin{equation} \label{assumption_B1}
B(|v-v_*|,\theta) = C|v-v_*|^{\gamma}\sin{\theta}\cos{\theta}
\end{equation}
for some $C>0$ and $0 \leq \gamma \leq 1$. The range of $\gamma$ corresponds to the hard potential cases. Throughout this article, we adopt a convenient notation: 
We denote $f\lesssim g$ if there exists a constant $C\geq 0$ such that $f \leq Cg$. Under the assumption \eqref{assumption_B1}, as have been observed by many authors, the linear operator $L(f)$ satisfies the following property.
\begin{PropertyA*}
The linear operator $L(f)$ can be decomposed as
\begin{equation*}
L(f)(x,v)=-\nu (v)f(x,v)+K(f)(x,v),
\end{equation*}
where
\begin{equation*}
K(f)(x,v):=\int_{\mathbb{R}^3}k(v,v_*)f(x,v_*)\,dv_*.
\end{equation*}
Here, $\nu (v)$ and $k(v,v_*)$ satisfy the following estimates
\begin{align}
\label{nu decay estimate}
& (1+|v|)^{\gamma} \lesssim \nu(v) \lesssim (1+|v|)^{\gamma},\\
\label{gradient nu decay estimate}
&|\nabla\nu(v)|\lesssim (1+|v|)^{\gamma - 1},\\
\label{K decay estimate}
&|k(v,v_*)|\lesssim \frac{1}{|v-v_*|(1+|v|+|v_*|)^{1-\gamma}}e^{-\frac{1-\rho}{4}\left(|v-v_*|^2+ \left( \frac{|v|^2-|v_*|^2}{|v-v_*|} \right)^2 \right)},\\
\label{gradient v K decay estimate}
&|\nabla_v k(v,v_*)|\lesssim \frac{1+|v|}{|v-v_*|^2(1+|v|+|v_*|)^{1-\gamma}}e^{-\frac{1-\rho}{4}\left(|v-v_*|^2+ \left( \frac{|v|^2-|v_*|^2}{|v-v_*|} \right)^2 \right)}
\end{align}
for $0 \leq \rho < 1$. 
\end{PropertyA*}
We remark that, under the assumption \eqref{nu decay estimate}, the function $\nu$ is uniformly positive;
\[
\inf_{v \in \R^3} \nu(v) > 0
\]
for all $0 \leq \gamma \leq 1$, which plays a key role in our analysis.

Next, we introduce the following notations:
\begin{align*}
\tau_{x,v}:=&\inf\{s\geq 0\mid x-sv\in \Omega^c\},\\   
q(x,v):=&x-\tau_{x,v}v,\\
N(x,v):=&-n(q(x,v))\cdot \frac{v}{|v|}.
\end{align*}
Also, we define
\begin{align*}
&| f |_{\infty,\alpha} := \esssup_{(x, v) \in \Omega \times \R^3} e^{\alpha |v|^{2}} |f(x, v)|, \\
&w(x,v) := \frac{|v|}{|v|+1}N(x,v),\\
&| f |_{\infty,\alpha, w}:= | w f |_{\infty, \alpha},\\
&\| f \|_{\infty,\alpha}:=| f|_{\infty,\alpha}+|\nabla_{x}f|_{\infty,\alpha, w}+|\nabla_{v}f|_{\infty,\alpha, w},\\
&L^{\infty}_{\alpha}:=\{f \mid | f |_{\infty,\alpha}<\infty\},\\
&\hat{L}^{\infty}_{\alpha}:=\{f \mid || f||_{\infty,\alpha}<\infty\},
\end{align*}
where $\alpha \geq 0$. 

We assume that the space domain $\Omega$ has the following property.
\begin{assumptionOmega*}
The space domain $\Omega$ is a bounded domain with $C^2$ boundary of positive Gaussian curvature.
\end{assumptionOmega*}
Notice that the positivity of Gaussian curvature implies non-vanishing principle curvatures.

In addition, we introduce uniform circumscribed and interior sphere conditions.

\begin{definition}    
Given $\Omega \subset \R^3 $, we say that the boundary of $\Omega$ satisfies the uniform circumscribed sphere condition if there exists a positive constant $R$ such that for any $x \in \partial \Omega$ there exists a ball $B_R$ with radius $R$ such that 
\[
x\in \partial B_R, \ \bar{\Omega} \subset \bar{B}_R.
\]
We call the constant $R$ the uniform circumscribed radius.
\end{definition}

\begin{definition}    
Given $\Omega \subset \R^3 $, we say that the boundary of $\Omega$ satisfies the uniform interior sphere condition if there exists a positive constant $r$ such that for any $x \in \partial \Omega $ there exists a ball $B_r$ with radius $r$ such that 
\[
x\in \partial B_r, \ \bar{B}_r\subseteq \overline{\Omega}.
\]
We call the constant $r$ the uniform interior radius.
\end{definition}

It is known that, if the domain $\Omega$ satisfies \textbf{Assumption $\Omega$}, then the boundary of $\Omega$ satisfies both uniform circumscribed and interior sphere conditions. For details, see \cite{Ikun 2}.      

The main result of this article is the following theorem.
\begin{theorem}\label{main theorem nonlinear}
Suppose \eqref{assumption_B1} holds. Given $0 \leq \alpha < (1-\rho)/2$, where $\rho$ is the constant in \textbf {Property A}, there exists a positive constant $\delta$ such that: For any domain $\Omega$ satisfying \textbf{Assumption $\Omega$} with uniform circumscribed and interior radii $R$ and $r$ respectively, if 
\begin{equation} \label{est_smallness}
\max \left\{ \diam(\Omega),(Rr)^{\frac{1}{2}} \left( 1+\frac{R}{r} \right), \| e^{-\nu(v)\tau_{x,v}}g(q(x,v),v) \|_{\infty,\alpha} \right\} < \delta,
\end{equation}
then the equation \eqref{Boltzmann equation} with the incoming boundary condition \eqref{boundary condition} admits a solution in $\hat{L}^{\infty}_{\alpha}$. 
\end{theorem}

\begin{remark}
There exist a domain and a boundary data which satisfy the assumption of Theorem \ref{main theorem nonlinear}. For example, let $\Omega:= B_{r_0}(0)$. Then, we have $R = r = r_0$ and $\diam(\Omega) = 2 r_0$. Also take $g = r_0 g_0$, where $g_0$ is a smooth function on $\Gamma^-$ with $|g_0(x,v)| \lesssim e^{-\alpha |v|^2}$, $|\nabla_X g_0(x,v)| \lesssim (1+|v|)e^{-\alpha|v|^2}/|v|$ and $ |\nabla_v g_0(x,v)|\lesssim (1+|v|)e^{-\alpha|v|^2}/|v|$. Here $\nabla_X$ denotes the covariant derivative of $g$ along the surface $\partial \Omega$.
\end{remark}

\begin{remark}
It can be shown that $\hat{L}^{\infty}_{\alpha}\subseteq W^{1,p}$, by which it turns out that Theorem \ref{main theorem nonlinear} provides a sufficient condition for the existence of a $W^{1,p}$ solution to the equation \eqref{Boltzmann equation} with the boundary condition \eqref{boundary condition} for $1 \leq p < 3$. A proof is given in Proposition \ref{appendix 1}.
\end{remark}

The strategy of our proof starts with the derivation of the existence of a solution to the linearized Boltzmann equation with a source term:
\begin{equation} \label{LBE}
\begin{cases}
v \cdot \nabla_{x}f + \nu(v)f = Kf+\phi, &(x,v) \in \Omega \times \mathbb{R}^3, \\ 
f(x,v)=g(x,v),  &(x,v) \in \Gamma^-.
\end{cases}    
\end{equation}
For the sake of the convenience in further discussion, we define operators $J$ and $S_\Omega$ by
\[
Jg(x, v) := e^{-\nu(v)\tau_{x,v}}g(q(x,v),v)
\]
and\[
S_{\Omega}h(x, v) := \int_0^{\tau_{x,v}}e^{-\nu(v)s} h(x - sv, v)\,ds,
\]
respectively. We notice that $Jg$ is a function defined on $\Omega \times \mathbb{R}^3$ while $g$ is defined on $\Gamma^-$. Then, the equation \eqref{LBE} is equivalent to the following integral form
\[
f(x,v) =e^{-\nu(v)\tau_{x,v}}g(q(x,v),v)+\int_0^{\tau_{x,v}}e^{-\nu(v)s}(Kf+\phi)(x-sv)\,ds,
\]
or
\begin{equation}\label{integral form}
f=Jg+S_{\Omega}Kf+S_{\Omega}\phi.
\end{equation}
By doing the Picard iteration on \eqref{integral form}, we obtain
\begin{equation}\label{Picard series finite chapter 1}
f=\sum_{i=0}^{n}(S_{\Omega}K)^{i}(Jg+S_{\Omega}\phi)+(S_{\Omega}K)^{n+1}f.
\end{equation}
We show the convergence of the summation on the right hand side of \eqref{Picard series finite chapter 1} as $n \to \infty$, and this limit gives a solution formula. Namely, the Picard iteration suggest the following solution formula to \eqref{integral form}.
\begin{equation}\label{Picard series}
f=\sum_{i=0}^{\infty}(S_{\Omega}K)^{i}(Jg+S_{\Omega}\phi).
\end{equation}

To achieve the above procedure, we need the following lemma.
\begin{lemma} \label{lem:iteration}
Assume {\bf Property A} and let $0 \leq \alpha < (1-\rho)/2$. Also, suppose {\bf Assumption $\Omega$} with uniform circumscribed and interior radii $R$ and $r$ respectively. Then, given $h \in \hat L^{\infty}_{\alpha}$, we have
 \begin{align*}
||(S_{\Omega}K)h||_{\infty,\alpha}\lesssim (1 + \diam(\Omega))|h|_{\infty,\alpha}+(Rr)^{\frac{1}{2}}\left( 1+\frac{R}{r} \right)||h||_{\infty,\alpha}.
\end{align*}   
\end{lemma}
Then, we reach at the following lemma.
\begin{lemma} \label{lem:existence linear}
Let $\phi$ be a function such that $S_\Omega \phi \in \hat L^{\infty}_{\alpha}$. Suppose \eqref{assumption_B1} holds. Then, given $0<\alpha<(1-\rho)/2$, where $\rho$ is the constant in \textbf{Property A}, there exists a positive constant $\delta$ such that: For any domain $\Omega$ satisfying \textbf{Assumption $\Omega$} with uniform circumscribed and interior sphere radii $R$ and $r$ respectively, if 
\begin{equation} \label{est_smallness_linear}
\max \left\{ \diam(\Omega),(Rr)^{\frac{1}{2}} \left( 1+\frac{R}{r} \right) \right\} <\delta,
\end{equation}
there exists a solution $f \in \hat{L}^{\infty}_{\alpha}$ to the integral equation \eqref{integral form}. Moreover, we have
\[
\Vert f \Vert_{\infty,\alpha} \lesssim \Vert S_{\Omega}\phi\Vert_{\infty,\alpha}+\Vert Jg \Vert_{\infty,\alpha}.
\]
\end{lemma}

Next, to obtain the solution to the equation \eqref{Boltzmann equation} with boundary condition \eqref{boundary condition}, we consider the following iteration scheme:
\begin{equation} \label{nonlinear ieration scheme}
\begin{cases}
v \cdot \nabla_{x} f_{i+1} + \nu(v) f_{i+1} = K(f_{i+1}) + \Gamma(f_{i},f_{i}), &(x, v) \in \Omega \times \R^3,\\
f_i(x,v)=g(x,v), &(x,v) \in \Gamma^-.
 \end{cases}
\end{equation}
To show the convergence of the sequence $\{f_i\}$ obtained by \eqref{nonlinear ieration scheme}, we employ the following lemma:
\begin{lemma} \label{lem:bilinear_est}
Suppose \eqref{assumption_B1} holds and let $0 \leq \alpha < (1-\rho)/2$. where $\rho$ is the constant in \textbf {Property A}. Also, suppose {\bf Assumption $\Omega$} with uniform circumscribed and interior radii $R$ and $r$ respectively. Then, for $h_1,h_2 \in \hat{L}^{\infty}_{\alpha}$, we have
 \begin{align*}
||S_{\Omega}\Gamma(h_1,h_2)||_{\infty,\alpha} \lesssim \left(1 + \diam(\Omega)+(Rr)^{\frac{1}{2}} \left( 1+\frac{R}{r} \right) \right)||h_1||_{\infty,\alpha} ||h_2||_{\infty,\alpha} . 
\end{align*}   
\end{lemma}
With the help of the bilinear estimate of $\| S_{\Omega}\Gamma(h_1,h_2) \|_{\infty,\alpha}$ in Lemma \ref{lem:bilinear_est}, we derive the convergence of the sequence $\{ f_i \}$. Finally we successfully prove that the limit of $\{ f_i \}$ is a solution to the equation \eqref{Boltzmann equation} with the boundary condition \eqref{boundary condition}.

The rest part of this article is as follows. In Section \ref{sec:pre}, we introduce some key estimates which are based on the assumption \eqref{assumption_B1} and the geometry of $\Omega$. In Section \ref{sec:linear}, we focus on the existence result and derive some estimates for the linear case. A detailed proof of the existence of a solution for the nonlinear case is given in Section 4.

\section{Preliminary estimates} \label{sec:pre}

In this section, we introduce some key estimates which are based on the assumption \eqref{assumption_B1} and the geometry of $\Omega$.

\subsection{Estimates for the linear integral kernel}

Thanks to the assumption \eqref{K decay estimate}, we can derive a very useful integral estimate which will be frequently used later.

\begin{lemma}\label{lem:K decay L estimate}
For $0 \leq \alpha < (1-\rho)/2$, where $\rho$ is the constant in {\bf Property A}, we have
\begin{equation*}
\int_{\R^3} \frac{1 + |v_*|}{|v_*|} |k(v,v_*)| e^{-\alpha|v_*|^2}\,dv_* \lesssim e^{-\alpha |v|^2} . 
\end{equation*}
\end{lemma}

Before starting our proof, we state a key estimate of $k$.
\begin{lemma}\label{key K estimate}
For any $\alpha \in \R$, we have
\begin{equation*}
\begin{split} 
|k(v,v_*)|\lesssim   &\frac{1}{|v-v_*|(1+|v|+|v_*|)^{1-\gamma}}e^{-\alpha|v|^2}  \\&\times e^{-\frac{(1-\rho+2\alpha)(1-\rho-2\alpha)}{4(1-\rho)}|v-v_*|^2}e^{-(1-\rho)\left( v\cdot\frac{(v-v_*)}{|v-v_*|}-\frac{1-\rho+2\alpha}{2(1-\rho)}|v-v_*| \right)^2}e^{\alpha|v_*|^2}.  
\end{split}
\end{equation*}   
\end{lemma}

\begin{proof}
Notice that $|v_*|^2=|v|^2+|v_*-v|^2+2v\cdot(v_*-v)$, which implies
\begin{equation*}
 \frac{|v|^2-|v_*|^2}{|v-v_*|}=-|v-v_*|+2v\cdot\frac{(v-v_*)}{|v-v_*|}.   
\end{equation*}
Hence by substituting above equation for \eqref{K decay estimate}, we have
\begin{align*}
e^{-\frac{1-\rho}{4}\left(|v-v_*|^2+ \left( -|v-v_*|+2v\cdot\frac{(v-v_*)}{|v-v_*|} \right)^2 \right)} =& e^{-\frac{1-\rho}{4}\left(2|v-v_*|^2-4v\cdot (v-v_*)+4\left( v\cdot\frac{(v-v_*)}{|v-v_*|} \right)^2 \right)}\\
=& e^{-\frac{1-\rho}{2}|v-v_*|^2}e^{(1-\rho)v\cdot (v-v_*)}e^{-(1-\rho)\left( v\cdot\frac{(v-v_*)}{|v-v_*|} \right)^2}.
\end{align*}
Then, we use the identity 
\[
|v|^2+|v-v_*|^2-2(v-v_*)\cdot v-|v_*|^2 = 0
\]
to obtain
\begin{equation*}
\begin{split}
&e^{-\frac{1-\rho}{2}|v-v_*|^2}e^{(1-\rho)v\cdot (v-v_*)}e^{-(1-\rho)\left( v\cdot\frac{(v-v_*)}{|v-v_*|} \right)^2}\\
=& e^{-\alpha(|v|^2+|v-v_*|^2-2(v-v_*)\cdot v-|v_*|^2)}e^{-\frac{1-\rho}{2}|v-v_*|^2}e^{(1-\rho)v\cdot (v-v_*)}e^{-(1-\rho)\left( v\cdot\frac{(v-v_*)}{|v-v_*|} \right)^2}\\
=& e^{-\alpha|v|^2}e^{-\frac{1-\rho+2\alpha}{2}|v-v_*|^2}e^{(1-\rho+2\alpha)v\cdot (v-v_*)}e^{-(1-\rho)\left( v\cdot\frac{(v-v_*)}{|v-v_*|} \right)^2}e^{\alpha|v_*|^2}\\
=& e^{-\alpha|v|^2}e^{-\frac{(1-\rho+2\alpha)(1-\rho-2\alpha)}{4(1-\rho)}|v-v_*|^2}e^{-(1-\rho)\left( v\cdot\frac{(v-v_*)}{|v-v_*|}-\frac{1-\rho+2\alpha}{2(1-\rho)}|v-v_*| \right)^2}e^{\alpha|v_*|^2}.
 \end{split}
\end{equation*}
This completes the proof.    
\end{proof}

\begin{proof}[Proof of Lemma \ref{lem:K decay L estimate}]
By applying Lemma \ref{key K estimate}, we have 
\begin{align*}
&\int_{\R^3} \frac{1 + |v_*|}{|v_*|} |k(v,v_*)| e^{-\alpha|v_*|^2}dv_*\\
\lesssim& \int_{\R^3} \frac{1 + |v_*|}{|v_*|} \frac{1}{|v-v_*|(1+|v|+|v_*|)^{1-\gamma}}e^{-\alpha|v|^2}e^{-\frac{(1-\rho+2\alpha)(1-\rho-2\alpha)}{4(1-\rho)}|v-v_*|^2}\\
&\times e^{-(1-\rho)\left( v\cdot\frac{(v-v_*)}{|v-v_*|}-\frac{1-\rho+2\alpha}{2(1-\rho)}|v-v_*| \right)^2}e^{\alpha|v_*|^2}e^{-\alpha|v_*|^2}dv_* \\
\leq& e^{-\alpha|v|^2}\int_{\R^3} \frac{1 + |v_*|}{|v_*|} \frac{1}{|v-v_*|(1+|v|+|v_*|)^{1-\gamma}}e^{-\frac{(1-\rho+2\alpha)(1-\rho-2\alpha)}{4(1-\rho)}|v-v_*|^2} dv_*. 
\end{align*}
Since $1-\rho-2\alpha > 0$, we get
\begin{align*}
&\int_{\R^3} \frac{1 + |v_*|}{|v_*|} \frac{1}{|v-v_*|(1+|v|+|v_*|)^{1-\gamma}}e^{-\frac{(1-\rho+2\alpha)(1-\rho-2\alpha)}{4(1-\rho)}|v-v_*|^2} dv_* \\
\leq& \int_{\R^3} \frac{1 + |v_*|}{|v_*|} \frac{1}{|v-v_*|}e^{-\frac{(1-\rho+2\alpha)(1-\rho-2\alpha)}{4(1-\rho)}|v-v_*|^2} dv_*\\
\leq& \int_{|v_*|\geq |v-v_*|} \frac{1 + |v_*|}{|v_*|} \frac{1}{|v-v_*|}e^{-\frac{(1-\rho+2\alpha)(1-\rho-2\alpha)}{4(1-\rho)}|v-v_*|^2} dv_*\\
&+ \int_{|v_*|\leq |v-v_*|} \frac{1 + |v_*|}{|v_*|} \frac{1}{|v-v_*|}e^{-\frac{(1-\rho+2\alpha)(1-\rho-2\alpha)}{4(1-\rho)}|v-v_*|^2} dv_*\\
\leq& \int_{|v_*|\geq |v-v_*|} \left( 1+\frac{1}{|v-v_*|} \right)\frac{1}{|v-v_*|}e^{-\frac{(1-\rho+2\alpha)(1-\rho-2\alpha)}{4(1-\rho)}|v-v_*|^2} dv_*\\
&+ \int_{|v_*|\leq |v-v_*|} \frac{1 + |v_*|}{|v_*|} \frac{1}{|v_*|}e^{-\frac{(1-\rho+2\alpha)(1-\rho-2\alpha)}{4(1-\rho)}|v_*|^2} dv_*\\
\leq& 2 \int_{\R^3} \frac{1 + |v_*|}{|v_*|^2} e^{-\frac{(1-\rho+2\alpha)(1-\rho-2\alpha)}{4(1-\rho)}|v_*|^2} dv_*\\
\lesssim& 1.
\end{align*}
Therefore, we obtain
\[
\int_{\R^3} \frac{1 + |v_*|}{|v_*|} |k(v,v_*)| e^{-\alpha|v_*|^2}dv_* \lesssim e^{- \alpha |v|^2}.
\]
This completes the proof. 
\end{proof}

We use the same idea to derive a similar estimate for $\nabla_v k(v,v_*)$ once we adopt \textbf{Property A}.
\begin{lemma}\label{lem:gradient v K decay L estimate}
For $0 \leq \alpha < (1-\rho)/2$, where $\rho$ is the constant in {\bf Property A}, we have
\begin{equation*}
\int_{\R^3}  |\nabla_v k(v,v_*)|e^{-\alpha|v_*|^2}dv_*\lesssim(1+|v|)^\gamma e^{-\alpha |v|^2}. 
\end{equation*}
\end{lemma}

\subsection{Estimates for the nonlinear cross section}
We provide three estimates for the nonlinear term whose cross section $B$ satisfies the assumption \eqref{assumption_B1}.

\begin{lemma} \label{lem:estimate_nonlin_1}
Let $\beta > 0$ and $0 \leq \gamma \leq 1$. Then, we have
\[
\int_{\R^3} \left(1+\frac{1}{|v_*|}\right) e^{-\beta|v_*|^2} |v - v_*|^\gamma\,dv_* \lesssim (1 + |v|)^\gamma
\]
for all $v \in \R^3$.
\end{lemma}

\begin{proof}
By the triangular inequality, we have
\begin{align*}
\int_{\R^3} \frac{1}{|v_*|^j} e^{-\beta |v_*|^2} |v - v_*|^\gamma\,dv_* \leq& |v|^\gamma \int_{\R^3} \frac{1}{|v_*|^j} e^{-\beta |v_*|^2} \,dv_* + \int_{\R^3} |v_*|^{\gamma - j} e^{-\beta |v_*|^2} \,dv_*\\
\lesssim& |v|^\gamma + 1\\
\lesssim& (1 + |v|)^\gamma
\end{align*}
for $j = 0, 1$. This completes the proof.
\end{proof}

\begin{lemma} \label{lem:estimate_nonlin_2}
Let $\beta > 0$ and $0 \leq \gamma \leq 1$. Then, we have
\[
\int_{\R^3} e^{-\beta|v_*|^2} |v - v_*|^{\gamma-1}\,dv_* \lesssim 1
\]
for all $v \in \R^3$.
\end{lemma}

\begin{proof}
We decompose the integral into two parts: 
\begin{align*}
&\int_{\R^3} e^{-\beta |v_*|^2} |v - v_*|^{\gamma-1}\,dv_*\\ 
=& \int_{|v - v_*| < |v_*|} e^{-\beta |v_*|^2} |v - v_*|^{\gamma-1}\,dv_* + \int_{|v - v_*| \geq |v_*|} e^{-\beta |v_*|^2} |v - v_*|^{\gamma-1}\,dv_*\\
\leq& \int_{|v - v_*| < |v_*|} e^{-\beta |v - v_*|} |v - v_*|^{\gamma-1}\,dv_* + \int_{|v - v_*| \geq |v_*|} e^{-\beta |v_*|^2} |v_*|^{\gamma-1}\,dv_*\\
\lesssim& \int_{\R^3} e^{-\beta |v_*|^2} |v_*|^{\gamma-1}\,dv_*\\
\lesssim& 1.
\end{align*}
This completes the proof.
\end{proof}

\begin{lemma} \label{lem:estimate_nonlin_3}
Let $\Omega$ be a bounded domain in $\R^3$ and $0 \leq \gamma \leq 1$. Then, we have
\[
(1 + |v|)^\gamma \min \left\{ 1, \frac{\diam(\Omega)}{|v|} \right\} \leq 1 + \diam(\Omega)
\]
for all $v \in \R^3 \setminus \{0\}$.
\end{lemma}

\begin{proof}
For $0 < |v| < \diam(\Omega)$, we have
\[
(1 + |v|)^\gamma \min \left\{ 1, \frac{\diam(\Omega)}{|v|} \right\} \leq (1 + \diam(\Omega) )^\gamma \leq 1 + \diam(\Omega).
\]
On the other hand, for $|v| \geq \diam(\Omega)$, we have
\begin{align*}
(1 + |v|)^\gamma \min \left\{ 1, \frac{\diam(\Omega)}{|v|} \right\} \leq& \frac{(1 + |v|)^\gamma}{|v|}\diam(\Omega)\\
\leq& \left( 1 + \frac{1}{|v|} \right)\diam(\Omega)\\
\leq& 1 + \diam(\Omega). 
\end{align*}
Here and in what follows, we use the estimate; $(1 + |v|)^\gamma \leq 1 + |v|$ for all $0 \leq \gamma \leq 1$ and $v \in \R^3$. Therefore, the estimate is proved.
\end{proof}

\subsection{Some geometrical estimates on bounded convex domains}

We introduce some geometric properties for bounded domains with $C^2$ boundaries with uniform circumscribed and interior radii $R$ and $r$ respectively. 

Based on such sphere conditions, we derive some important estimates.

\begin{prop} \label{prop:geometric estimate C}
Given $\Omega$ with the uniform circumscribed radius $R$, we have
\[
|x-q^+(x,v)|\lesssim R N(x,v)
\]
for all $v \in \mathbb{R}^3 \setminus \{0 \}$ and $x \in \Gamma^-_v$, where $\Gamma^-_v := \{ x \in \partial \Omega \mid n(x) \cdot v < 0 \}$ and $q^+(x, v) := q(x, -v)$. 
\end{prop}

\begin{proof}
We fix $v \in \mathbb{R}^3 \setminus \{0 \}$ and $x \in \Gamma^-_v$. By the definition of the uniform circumscribed radius, there exists a ball $B_R$ with radius $R$ such that $\partial B_R$ intersects $\partial \Omega$ at $x$. Let $O$ be the center of the ball $B_R$ and let $A$ be the intersection point of the half-line $\overrightarrow{x q^+}$ and $\partial B_R$, as Figure \ref{Math08.png} shows.

Let $\theta$ be the angle between $n(x)$ and $-v$. We note that the normal $n(x)$ to $\partial \Omega$ is also the normal to $\partial B_R$ at $x$. By a geometrical observation, it is clear that 
\[
|x - q^+(x,v)| \leq \overline{xA} \leq 2R\cos{\theta}=2RN(x,v).
\]
This completes the proof.

\begin{figure}[ht]
\centering
\includegraphics[width=0.8\linewidth]{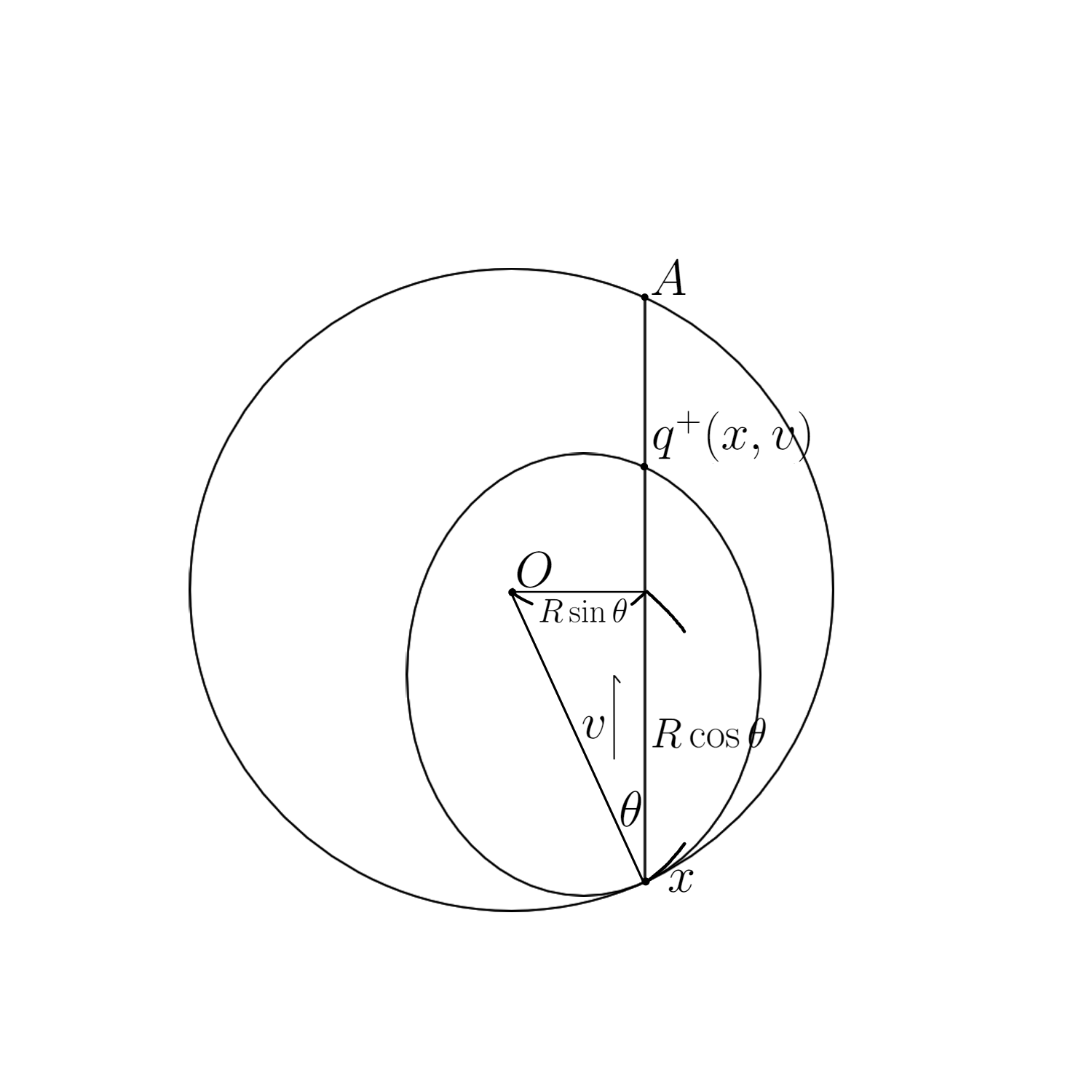}  
\caption{A picture of the cross section of $\Omega$ and $B_R$ containing $O$, $x$ and $q^+(x,v)$ in Proposition \ref{prop:geometric estimate C}.}
\label{Math08.png}
\end{figure}

\end{proof}

\begin{prop} \label{prop:geometric estimate A}
Given $\Omega$ with the uniform circumscribed radius $R$, we have
\[
d_x^{\frac{1}{2}}\leq R^{\frac{1}{2}} N(x,v)
\]
for all $x \in \overline{\Omega}$ and $v \in \mathbb{R}^3 \setminus \{0 \}$. Here, $d_x$ denotes the distance between $x$ and $\partial \Omega$.
\end{prop}

\begin{proof}
By the uniform circumscribed sphere condition, for any $x \in \Omega$ and $v \in \R^3 \setminus \{0\}$, there exists a ball $B_R \supseteq \Omega$ with radius $R$ such that $q(x,v) \in \partial B_R \cap \partial \Omega$. Let $O$ be the center of the ball $B_R$.

We first consider the case where three points $O$, $x$ and $q(x,v)$ are collinear. In this case, since $\partial B_R$ shares its normal vector with $\partial \Omega$ at $q(x, v)$, we have $N(x, v)=1$ and the desired estimate is obviously true. Thus, in what follows, we assume that $O$, $x$ and $q(x,v)$ are not collinear. Moreover, we discuss the estimate in the two-dimensional section since these three points are on the same plane. 

Let $A$ be the intersection point of $\partial B_R$ and the half line $\overrightarrow{Ox}$. Take a point $B$ on the line $xq$ such that the line $OB$ is perpendicular to it. We notice that the point $B$ is the closest to $O$ among the points on the line $Oq$. Thus, the inequality $\overline{OB} \leq \overline{Ox}$ holds. Let $C$ be defined as the intersection point of $\partial B_R$ and the half line $\overrightarrow{OB}$. We provide a picture of the configuration as Figure \ref{Math04.png}. 

\begin{figure}[ht]
\centering
\includegraphics[width=0.8\linewidth]{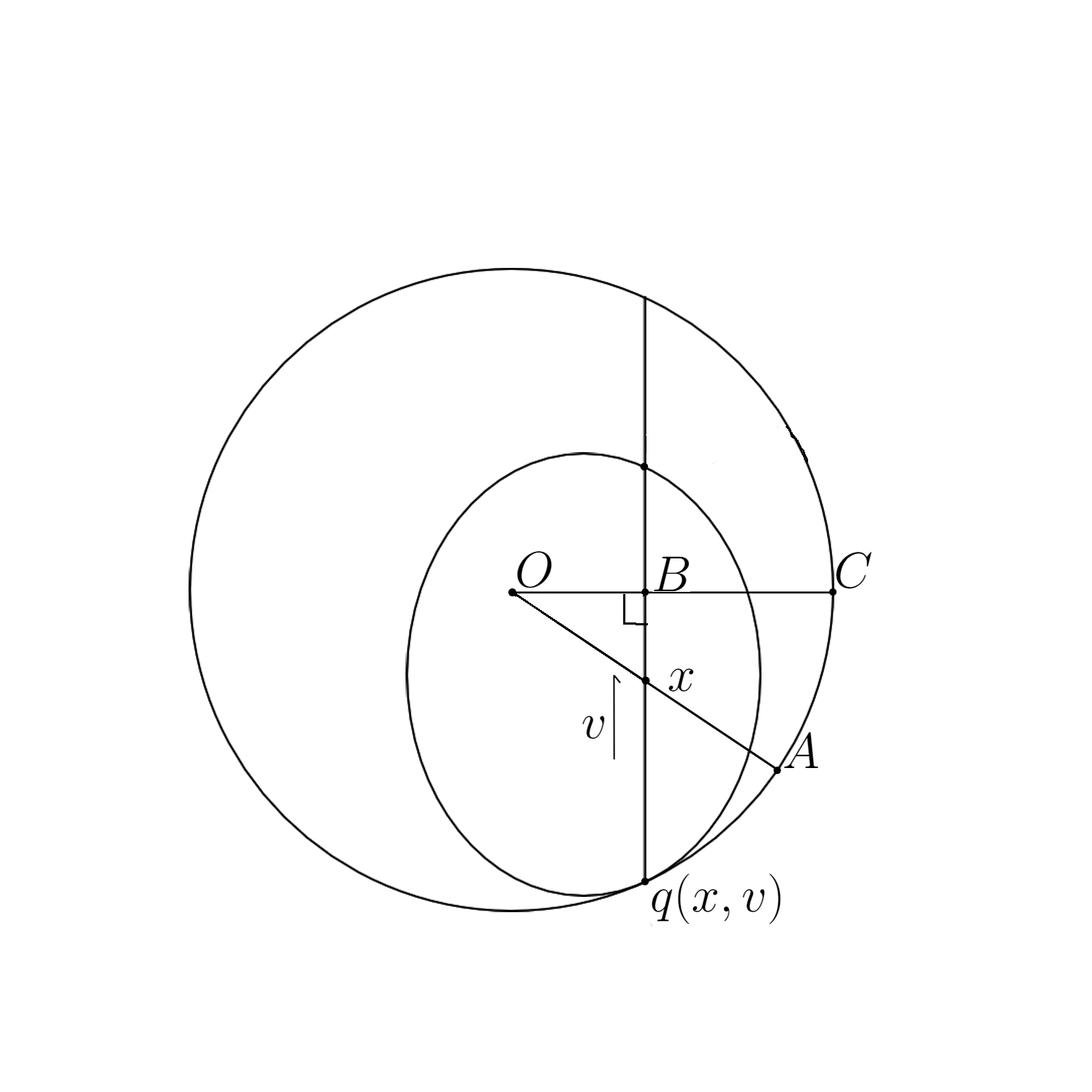}  
\caption{A picture of the cross section of $\Omega$ and $B_R$ containing $O$, $x$ and $q(x,v)$ in Proposition \ref{prop:geometric estimate A}.}
\label{Math04.png}
\end{figure}

We are ready to prove the estimate. Thanks to the condition $\Omega \subset B_R $, we have
\begin{equation*}
    d_x \leq d(x,\partial B_R).
\end{equation*}
Here, $d(x, X)$ denotes the distance between the point $x$ and the set $X$. We notice that 
\[
d(x,\partial B_R) = \overline{xA} = R - \overline{Ox} \leq R - \overline{OB} = \overline{BC}.
\]

In the same way as in the proof of Proposition \ref{prop:geometric estimate C}, let $\theta$ be the angle between $n(q(x, v))$ and $-v$. Then, as Figure \ref{Math05.png} illustrates, noticing that $N(x, v) = \cos \theta$, we have 
\begin{equation*}
\overline{BC} = R (1-\sin{\theta}) = R (1-\sqrt{1-\cos^2{\theta}}) = R (1-\sqrt{1-N^2(x,v)}) \leq RN^2(x,v).    
\end{equation*}

Combining the above estimates, we obtain
\[
d_x \leq RN^2(x,v).
\]
This completes the proof.

\begin{figure}[t]
\centering
\includegraphics[width=0.8\linewidth]{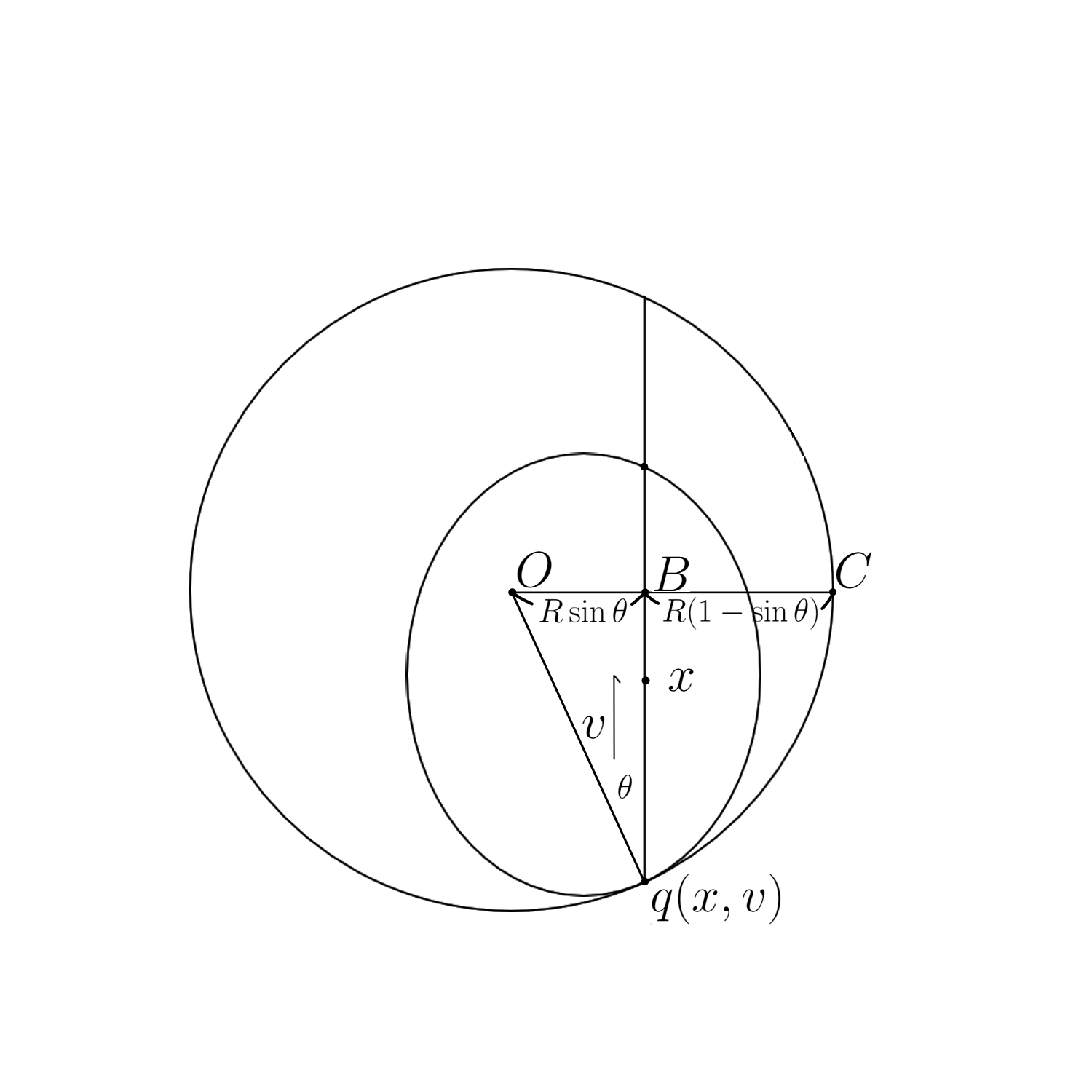}  
\caption{A picture of $O$, $B$, $C$ and $q(x,v)$ in Proposition \ref{prop:geometric estimate A} (deleting the point $A$).}
\label{Math05.png}
\end{figure}

\end{proof}

\begin{prop} \label{prop:geometric estimate B}
Given $\Omega$ with uniform circumscribed and interior radii $R$ and $r$ respectively, we have
\begin{equation*}
\int_{0}^{|q(x,v)-x|}\frac{1}{d_{x-s\frac{v}{|v|}}^{\frac{1}{2}}}\,ds \lesssim r^{\frac{1}{2}}+\frac{R}{r^{\frac{1}{2}}},
\end{equation*}
for all $x \in \overline{\Omega}$ and $v \in \R^3 \setminus \{0\}$.
\end{prop}

\begin{remark}
A proof of Proposition \ref{prop:geometric estimate B} was firstly explored as that of Lemma 5.12 in \cite{Ikun 2}. We repeat the proof for readers' convenience. 
\end{remark}
 
\begin{proof}
Without loss of generality, we assume that $x \in \Gamma ^{+}_v:=\{x\in \partial \Omega \mid n(x)\cdot v >0\}$.

We first consider the case where the domain is the ball with center $O$ and radius $r$. Let $A$ be the mid point of $x$ and $q(x,v)$. 

If $O = A$, then we have 
\[
\int_0^{|q(x, v) - x|} \frac{1}{d_{x - s\frac{v}{|v|}}}\,ds = 2 \int_0^r \frac{dt}{t^{\frac{1}{2}}} = 4 r^{\frac{1}{2}},
\]
which implies that the desired estimate holds. In what follows, we discuss the case where $O \neq A$. 

Let $B$ be the intersection point of $\partial B_R$ and the half line $\overrightarrow{OA}$. Take a point $C := x-sv/|v|$, $0 < s < |x-q(x,v)|$ and let $D$ be the intersection point between the line segment $xB$ and the line passing through $C$ and parallel to the line $OB$. The configuration of these points is seen in Figure \ref{Math06.png}.

\begin{figure}[ht]
\centering
\includegraphics[width=0.8\linewidth]{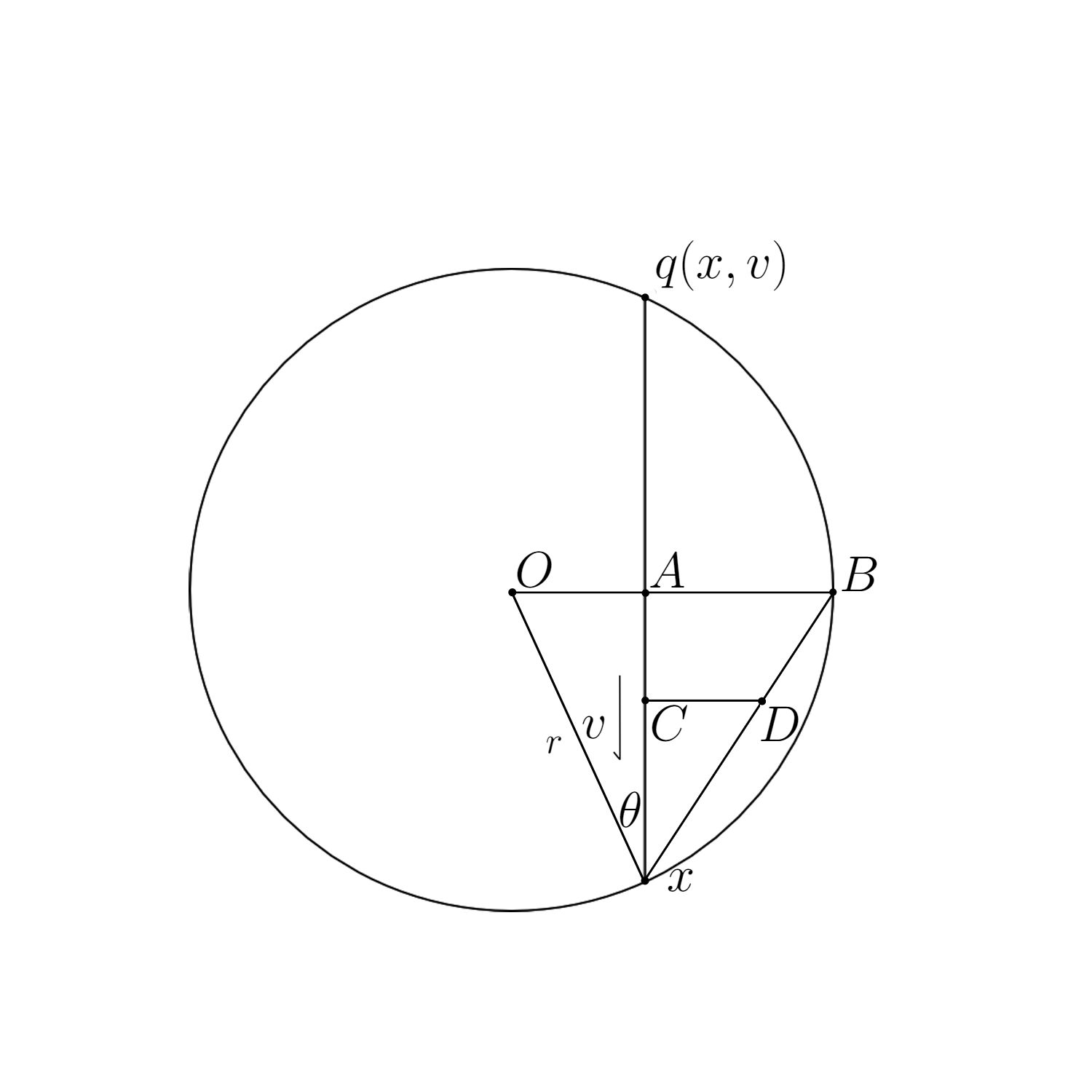}
\caption{A picture of points $O$, $A$, $B$, $C$ and $D$ in Proposition \ref{prop:geometric estimate B} when the domain is a ball.}
\label{Math06.png}
\end{figure}

We claim that $d(C, \partial B_R) \geq \overline{CD}/\sqrt{2}$. Let $\theta$ be the angle between $n(x)$ and $v$, which is $\angle OxA$. Through a geometrical observation, we see that $\angle xDC = \angle xBA = \pi/4 + \theta/2$. Thus, we have
\[
d(C, \partial B_R) \geq \overline{CD} \sin \left( \frac{\pi}{4} + \frac{\theta}{2} \right) \geq \frac{\overline{CD}}{\sqrt{2}}.
\]

Since $\triangle xAB$ is similar to $\triangle x CD$, we have 
\begin{equation*}
\overline{CD} = \overline{AB}\cdot \frac{\overline{Cx}}{Ax}=r(1-\sin{\theta})\frac{\overline{Cx}}{r\cos{\theta}}=(1-\sin{\theta})\frac{\overline{Cx}}{\cos{\theta}}.
\end{equation*}
Hence the integral in the left hand side is estimated as follows:
\begin{align*}
\int_{0}^{|q(x,v)-x|}\frac{1}{d_{x-s\frac{v}{|v|}}^{\frac{1}{2}}}\,ds & \leq 2\int_{0}^{r\cos{\theta}}\frac{1}{ \left( \frac{1}{\sqrt{2}}\frac{1-\sin{\theta}}{\cos{\theta}}t \right)^{\frac{1}{2}}}\,dt\\
& \lesssim \left( \frac{\cos{\theta}}{1-\sin{\theta}} \right)^{\frac{1}{2}}\int_{0}^{r\cos{\theta}}\frac{1}{t^{\frac{1}{2}}}\,dt\\
& \lesssim \left( \frac{\cos{\theta}}{1-\sin{\theta}} \right)^{\frac{1}{2}}(r\cos{\theta})^{\frac{1}{2}} \\
& = \left( \frac{\cos^2{\theta}}{1-\sin{\theta}} \right)^{\frac{1}{2}}r^{\frac{1}{2}} \\
& = (1+\sin{\theta})^{\frac{1}{2}}r^{\frac{1}{2}} \\
& \lesssim r^{\frac{1}{2}},
\end{align*}
which is the desired estimate.    

We proceed to discuss the case for a general domain with the uniform circumscribed condition with radius $R$. Based on the uniform interior sphere condition with radius $r$, we introduce two open balls $B_1$ and $B_2$ with radii $r$ such that $x\in \partial B_1$, $\bar{B}_1 \subseteq \overline{\Omega}$ and $q(x,v)\in \partial B_2$, $\bar{B}_2 \subseteq \overline{\Omega}$.

If the line segment $xq$ is in $\overline{B_1} \cup \overline{B_2}$, then we have
\begin{align*}
&\int_0^{|q(x, v) - x|} \frac{1}{d_{x - s\frac{v}{|v|}}^{\frac{1}{2}}}\,ds\\ 
\leq& \int_{\{ x - s\frac{v}{|v|} \in B_1 \}} \frac{1}{d \left( x - s\frac{v}{|v|}, \partial B_1 \right)^{\frac{1}{2}}}\,ds + \int_{\{ x - s\frac{v}{|v|} \in B_2 \}} \frac{1}{d \left( q(x, v) - s\frac{v}{|v|}, \partial B_2 \right)^{\frac{1}{2}}}\,ds\\
\lesssim& r^{\frac{1}{2}}.
\end{align*}

If it is not the case, the line segment $xq$ is still in the convex hull of $B_1 \cup B_2$. Let $\theta_1$ is the angle between $n(x)$ and $v$, and let $\theta_2$ be that between $q(x, v)$ and $-v$. Then, we have $d_{x - sv/|v|} \geq r \min \{ 1 - \sin \theta_1, 1 - \sin \theta_2 \}$ when $x - sv/|v| \notin B_1 \cup B_2$. Also, by switching the pair $(x, v)$ and $(q^+(x, v), -v)$ in the proof of Proposition \ref{prop:geometric estimate C}, we can see that the length of the line segment $\{ x - sv/|v| \mid 0 < s < |q(x, v) - x|, x - sv/|v| \notin B_1 \cup B_2 \}$ is less than 
\[|x - q(x, v)| \lesssim R \min \{ \cos \theta_1, \cos \theta_2 \}. 
\]
Thus, we have
\begin{align*}
&\int_0^{|q(x, v) - x|} \frac{1}{d_{x - s\frac{v}{|v|}}^{\frac{1}{2}}}\,ds\\ 
\leq& \int_{\{ x - s\frac{v}{|v|} \in B_1 \}} \frac{1}{d \left( x - s\frac{v}{|v|}, \partial B_1 \right)^{\frac{1}{2}}}\,ds + \int_{\{ x - s\frac{v}{|v|} \in B_2 \}} \frac{1}{d \left( q(x, v) - s\frac{v}{|v|}, \partial B_2 \right)^{\frac{1}{2}}}\,ds\\ 
&+ \int_{\{ x - s\frac{v}{|v|} \notin B_1 \cup B_2 \}} \frac{1}{r^{\frac{1}{2}} \min \{ 1 - \sin \theta_1, 1 - \sin \theta_2 \}^{\frac{1}{2}}}\,ds\\
\lesssim& r^{\frac{1}{2}} + \frac{R}{r^{\frac{1}{2}}}.
\end{align*}
Here, we used facts that $\min \{ 1 - \sin \theta_1, 1 - \sin \theta_2 \} = 1 -\sin \theta_j$ if and only if $\min \{ \cos \theta_1, \cos \theta_2 \} = \cos \theta_j$ and that $0 \leq \cos \theta_j/(1 - \sin \theta_j)^{1/2} \lesssim 1$. 

This completes the proof.

\end{proof}

By changing variable of integration by $s = |v|t$ in Proposition \ref{prop:geometric estimate B}, we obtain the following estimate. 

\begin{corollary} \label{cor:geometric estimate B}
Given $\Omega$ with uniform circumscribed and interior radii $R$ and $r$ respectively, we have
\begin{equation*}
\int_{0}^{\tau_{x, v}} \frac{1}{d_{x-tv}^{\frac{1}{2}}}\,dt \lesssim \frac{r^{\frac{1}{2}}}{|v|} \left( 1 + \frac{R}{r} \right),
\end{equation*}
for all $x \in \overline{\Omega}$ and $v \in \R^3 \setminus \{0\}$.
\end{corollary}

\section{Regularity for the linearized case} \label{sec:linear}

In this section, we provide a detailed proof of the existence of a solution to the integral equation \eqref{integral form}. To make sure that the series \eqref{Picard series} converges in $\hat{L}_{\alpha}^{\infty}$, we need to estimate the $\hat{L}_{\alpha}^{\infty}$ norm for each term of the series.

We first present an important lemma.
\begin{lemma}\label{SK decay L infty estimate}
Let $\Omega$ be a bounded domain in $\R^3$ and suppose {\bf Property A}. Also, let $0 \leq \alpha < (1 - \rho)/2$. Then, given $h \in L^\infty_\alpha$, we have
\begin{align*}
|S_{\Omega}Kh(x,v)| \lesssim& |h|_{\infty,\alpha}e^{-\alpha |v|^{2}}\min \left\{ 1,\frac{\diam(\Omega)}{|v|} \right\}
\end{align*}
for a.e. $(x, v) \in \Omega \times \R^3$.
\end{lemma}

\begin{proof}
We start from the following estimate:
\begin{equation*}
\begin{split}
|S_{\Omega}Kh(x,v)| =& \left|\int_0^{\tau_{x,v}}e^{-\nu(v)s}Kh(x-vs,v)ds \right|\\
\lesssim& \int_0^{\tau_{x,v}}e^{-\nu(v)s}\left|Kh(x-vs,v)\right|ds.
\end{split}
\end{equation*}
Recall that 
\[
|h(x, v)| \lesssim |h|_{\infty, \alpha} e^{-\alpha |v|^2}
\]
for a.e. $(x, v) \in \Omega \times \R^3$. Thus, by Lemma \ref{lem:K decay L estimate}, we have
\begin{equation*}
|Kh(x-sv,v)|\lesssim |h|_{\infty, \alpha} e^{-\alpha|v|^{2}}.
\end{equation*}

We also recall that, by the assumption \eqref{nu decay estimate} in {\bf Property A}, the function $\nu$ is uniformly positive. Thus, we have 
\begin{equation} \label{est:S infty}
\int_{0}^{\tau_{x,v}}e^{-\nu(v)s}\,ds \lesssim \min \left\{ 1, \frac{\diam(\Omega)}{|v|} \right\},
\end{equation} 
and
\begin{equation*}
\begin{split}
\int_0^{\tau_{x,v}}e^{-\nu(v)s}\left|Kh(x-vs,v)\right|\,ds \lesssim& |h|_{\infty, \alpha} e^{-\alpha|v|^{2}} \int_0^{\tau_{x,v}}e^{-\nu(v)s}\,ds\\
\lesssim& |h|_{\infty, \alpha} e^{-\alpha|v|^{2}} \min \left\{ 1, \frac{\diam(\Omega)}{|v|} \right\}.
\end{split}
\end{equation*}
This completes the proof.    
\end{proof}

Lemma \ref{SK decay L infty estimate} gives the following estimate.

\begin{corollary}\label{cor:SK intergral estimate}
Let $\Omega$ be a bounded domain in $\R^3$ and suppose {\bf Property A}. Also, let $0 \leq \alpha < (1 - \rho)/2$. Then, given $h \in L^\infty_\alpha$, we have
\begin{align*}
|S_{\Omega}KS_{\Omega}Kh|_{\infty,\alpha} \lesssim& \diam(\Omega) |h|_{\infty,\alpha}.
\end{align*}
\end{corollary}

\begin{proof}
Let $h \in L^\infty_\alpha$. By Lemma \ref{SK decay L infty estimate}, we have 
\[
|S_\Omega K h(x, v)| \lesssim |h|_{\infty, \alpha} e^{-\alpha|v|^{2}} \frac{\diam(\Omega)}{|v|}
\]
for a.e. $(x, v) \in \Omega \times \R^3$. By Lemma \ref{lem:K decay L estimate} again, we have
\[
|K S_\Omega K h(x, v)| \lesssim \diam(\Omega) |h|_{\infty, \alpha} \int_{\R^3} \frac{1}{|v_*|} |k(v, v_*)| e^{-\alpha|v_*|^{2}} \lesssim \diam(\Omega) |h|_{\infty, \alpha}
\]
for a.e. $(x, v) \in \Omega \times \R^3$. The conclusion follows from the estimate \eqref{est:S infty}.
\end{proof}

For the estimate on the $x$ derivative, we need to use some geometric properties of $\Omega$. Recall that we assume uniform sphere conditions.  
\begin{lemma}\label{SK decay x gradient L infty estimate}
Suppose {\bf Property A} and let $0 \leq \alpha < (1 - \rho)/2$. Also, suppose {\bf Assumption $\Omega$} with uniform circumscribed and interior radii $R$ and $r$ respectively. Then, given $h \in \hat{L}^{\infty}_{\alpha}$, we have the following estimate:
\begin{equation*}
|\nabla_x S_{\Omega}Kh|_{\infty,\alpha,w} \lesssim |h|_{\infty,\alpha}+ (Rr)^{\frac{1}{2}} \left( 1+\frac{R}{r} \right) \| h \|_{\infty,\alpha}.
\end{equation*}
\end{lemma}

\begin{proof}
It is known in \cite{Eso 1} that
\begin{equation*}
\vert \nabla_{x}\tau_{x,v} \vert = \frac{\vert  -n(q(x,v)) \vert }{\vert v \vert N(x,v)}= \frac{1}{\vert v \vert N(x,v)}.
\end{equation*}
Also, by the definition of $S_{\Omega}$, we have
\[
\nabla_{x}S_{\Omega}Kh(x, v) = S_{\Omega, x} K h(x, v) + S_\Omega K \nabla_x h(x, v),
\]
where
\[
S_{\Omega, x} h(x, v) := -\frac{n(q(x, v))}{|v| N(x, v)} e^{-\nu(v) \tau_{x, v}} h(q(x, v), v).
\]

For the former term, we use Lemma \ref{lem:K decay L estimate} to derive
\begin{equation*}
\left| S_{\Omega, x} K h(x, v) \right|
\lesssim\frac{1}{|v|N(x,v)}|h|_{\infty, \alpha}e^{-\alpha|v|^{2}}.
\end{equation*}
For the latter term, we have
\begin{equation*}
\begin{split}
|S_{\Omega} K \nabla_{x}h(x, v)| &=\left|\int_{0}^{\tau_{x,v}}e^{-\nu(v)s}K \nabla_{x} h(x-vs,v)\,ds\right|\\
&\leq\int_{0}^{\tau_{x,v}}e^{-\nu(v)s}\left|K \nabla_{x} h(x-vs,v)\right|\,ds.
\end{split}
\end{equation*}
Since by definition of $| \cdot |_{\infty, \alpha, w}$, we deduce that
\begin{equation*}
|\nabla_{x} h(x, v)| \leq |\nabla_{x} h|_{\infty,\alpha,w} w(x, v)^{-1} e^{-\alpha|v|^{2}} \leq \| h \|_{\infty,\alpha} w(x, v)^{-1} e^{-\alpha|v|^{2}}.
\end{equation*}
Here, we recall that $w(x, v) = |v| N(x, v)/(1 + |v|)$. Thus, we have
\begin{equation*}
\begin{split}
|K \nabla_{x} h (x, v)| &=\left|\int_{\mathbb{R}^{3}}k(v,v_{*})\nabla_{x} h(x,v_{*})dv_{*}\right|\\
&\leq \| h \|_{\infty,\alpha} \int_{\mathbb{R}^{3}}|k(v,v_{*})| w(x, v_*)^{-1} e^{-\alpha|v_{*}|^{2}}dv_{*}.
\end{split}
\end{equation*}
By Lemma \ref{lem:K decay L estimate} and Proposition \ref{prop:geometric estimate A}, we get
\begin{equation*}
\begin{split}
\int_{\mathbb{R}^{3}}|k(v,v_{*})| w(x, v_*)^{-1} e^{-\alpha|v_{*}|^{2}}dv_{*}\lesssim& \frac{R^\frac{1}{2}}{d_{x}^{\frac{1}{2}}}\int_{\mathbb{R}^{3}} \frac{1 + |v_*|}{|v_{*}|} |k(v,v_{*})| e^{-\alpha|v_{*}|^{2}}dv_{*}\\
\lesssim& \frac{R^\frac{1}{2}}{d_{x}^{\frac{1}{2}}}e^{-\alpha|v|^{2}}.
\end{split}
\end{equation*}
Hence, using Corollary \ref{cor:geometric estimate B}, we obtain
\begin{equation*}
\begin{split}
|S_{\Omega} K \nabla_{x} h(x, v)| \lesssim& \int_{0}^{\tau_{x,v}}e^{-\nu(v)s} | \nabla_x h |_{\infty,\alpha,w} \frac{R^{\frac{1}{2}}}{d_{x-sv}^{\frac{1}{2}}}e^{-\alpha|v|^{2}}\,ds\\
\leq& R^{\frac{1}{2}} \| h \|_{\infty,\alpha} e^{-\alpha|v|^{2}} \int_{0}^{\tau_{x,v}} \frac{1}{d_{x-sv}^{\frac{1}{2}}}\,ds\\
\lesssim& w(x, v)^{-1} (Rr)^{\frac{1}{2}} \left( 1 + \frac{R}{r} \right) \| h \|_{\infty,\alpha} e^{-\alpha|v|^{2}}.
\end{split}
\end{equation*}
Therefore, we have
\begin{equation*}
| S_{\Omega} K \nabla_{x} h |_{\infty,\alpha,w} \lesssim (Rr)^{\frac{1}{2}} \left( 1+\frac{R}{r} \right) \| h \|_{\infty,\alpha}.
\end{equation*}
Hence, we conclude that 
\begin{equation*}
|\nabla_x S_{\Omega}Kh|_{\infty,\alpha,w} \lesssim |h|_{\infty,\alpha}+ (Rr)^{\frac{1}{2}} \left( 1+\frac{R}{r} \right) \| h \|_{\infty,\alpha},
\end{equation*}
which completes the proof.
\end{proof}

For the $v$ derivatives, we have the following estimate.
\begin{lemma}\label{SK decay v gradient L infty estimate}
Suppose {\bf Property A} and let $0 \leq \alpha < (1 - \rho)/2$. Also, suppose {\bf Assumption $\Omega$} with uniform circumscribed and interior radii $R$ and $r$ respectively. Then, given $h \in \hat{L}^{\infty}_{\alpha}$ , we have the follow estimate:
\begin{equation*}
|\nabla_v S_{\Omega}Kh|_{\infty,\alpha,w}\lesssim (1 + \diam(\Omega))|h|_{\infty,\alpha}+ (Rr)^{\frac{1}{2}} \left( 1+\frac{R}{r} \right) \| h \|_{\infty,\alpha}.
\end{equation*}
\end{lemma}

\begin{proof}
Notice that
\begin{align*}
\nabla_{v}S_{\Omega}Kh(x, v) =& S_{\Omega, v} K h(x, v) - (\nabla_v \nu(v)) S_{\Omega, s} K h(x, v)\\ 
&- S_{\Omega, s} K \nabla_x h(x, v) + S_\Omega K_v h(x, v),
\end{align*}
where
\begin{align*}
S_{\Omega, v} h(x, v) := &(\nabla_{v}\tau_{x,v}) e^{-\nu(v)\tau_{x,v}} h(q(x,v),v),\\
S_{\Omega, s} h(x, v) := &\int_{0}^{\tau_{x,v}} e^{-\nu(v)s} s Kh(x-sv,v)\,ds,\\
K_v h(x, v) :=& \int_{\mathbb{R}^{3}} \nabla_{v}k(v,v_{*}) h(x,v_{*})\,dv_{*}.
\end{align*}

For the first term in the right hand side, we notice that
\begin{equation*}
\vert \nabla_{v}\tau_{x,v} \vert = \frac{\vert x-q(x,v) \vert \vert  n(q(x,v)) \vert }{\vert v \vert^{2}N(x,v)}= \frac{\vert x-q(x,v) \vert  }{\vert v \vert^{2}N(x,v)}.
\end{equation*}
As a result, we have
\begin{equation*}
\begin{split} 
|S_{\Omega, v} K h(x, v)| \leq& \frac{\vert x-q(x,v) \vert}{\vert v \vert^{2}N(x,v)}e^{-\nu(v)\tau_{x,v}}|Kh(q(x,v),v)|\\
\lesssim& |h|_{\infty,\alpha}\frac{1}{|v| N(x,v)\nu(v)}\frac{\nu(v) |x-q(x,v)|}{|v|}e^{-\nu(v)\frac{|x-q(x,v)|}{|v|}}e^{-\alpha|v|^{2}}\\
\lesssim& |h|_{\infty,\alpha}\frac{1}{|v| N(x,v)}e^{-\alpha|v|^{2}}.
\end{split}
\end{equation*}
Here, we used the estimate \eqref{nu decay estimate} in order to guarantee that $\nu(v)^{-1} \lesssim 1$ for all $v \in \R^3$.

For the second and the third terms, we notice that $e^{-\nu(v)s}s \lesssim e^{-\nu(v)s/2}$ for all $s \geq 0$ and $v \in \R^3$ due to the property \eqref{nu decay estimate}. Thus, they are estimated in the same way as in the proof of Lemma \ref{SK decay L infty estimate} and Lemma \ref{SK decay x gradient L infty estimate} to yield
\[
|(\nabla_v \nu(v)) S_{\Omega, s} K h(x, v)| \lesssim |h|_{\infty, \alpha} e^{-\alpha |v|^2}
\]
and
\[
|S_{\Omega, s} K \nabla_x h(x, v)| \lesssim w(x, v)^{-1} \| h \|_{\infty,\alpha}(Rr)^{\frac{1}{2}} \left( 1+\frac{R}{r} \right) e^{-\alpha|v|^{2}}.
\]

For the last term, we recall Lemma \ref{lem:gradient v K decay L estimate} to control the integral:
\begin{equation*}
\begin{split} 
|S_\Omega K_v h(x, v)| \lesssim& |h|_{\infty,\alpha}\int_{0}^{\tau_{x,v}}e^{-\nu(v)s} \int_{\mathbb{R}^{3}}|\nabla_{v}k(v,v_{*})|e^{-\alpha|v_{*}|^{2}}dv_{*}ds\\
\lesssim& |h|_{\infty,\alpha}e^{-\alpha|v|^{2}}(1+|v|)^{\gamma}\int_{0}^{\tau_{x,v}}e^{-\nu(v)s}ds\\
\leq& |h|_{\infty,\alpha}e^{-\alpha|v|^{2}} \frac{1 + |v|}{|v|} \diam(\Omega).
\end{split}
\end{equation*}

The conclusion follows from the above estimates. We note that $|N(x, v)| \leq 1$ for all $(x, v) \in \Omega \times (\R^3 \setminus \{0\})$.
\end{proof}

We are ready to give a proof of Lemma \ref{lem:existence linear}. Consider the integral form of \eqref{integral form}. By using Picard iteration, we already derived a formal solution as the series \eqref{Picard series}. By Lemma \ref{lem:K decay L estimate}, \ref{SK decay L infty estimate}, \ref{SK decay x gradient L infty estimate}, and \ref{SK decay v gradient L infty estimate}, we derive the following estimate:
\begin{align*}
||(S_{\Omega}K)^i(Jg+S_{\Omega}\phi)||_{\infty,\alpha} \lesssim& (1 + \diam(\Omega))|(S_{\Omega}K)^{i-1}(Jg+S_{\Omega}\phi)|_{\infty,\alpha}\\
&+(Rr)^{\frac{1}{2}} \left( 1+\frac{R}{r} \right)||(S_{\Omega}K)^{i-1}(Jg+S_{\Omega}\phi)||_{\infty,\alpha}
\end{align*}
for all $i\geq 1$. With $\delta > 0$ in \eqref{est_smallness_linear} small enough, we have
\begin{equation*}
\begin{split}
&||(S_{\Omega}K)^i(Jg+S_{\Omega}\phi)||_{\infty,\alpha}-\frac{1}{2}||(S_{\Omega}K)^{i-1}(Jg+S_{\Omega}\phi)||_{\infty,\alpha}\\
\lesssim& (1+\diam(\Omega))|(S_{\Omega}K)^{i-1}(Jg+S_{\Omega}\phi)|_{\infty,\alpha}.
\end{split}
\end{equation*}
Therefore,
\begin{equation*}
\begin{split}
&\sum_{i=0}^{n}| |(S_{\Omega}K)^i(Jg+S_{\Omega}\phi)| |_{\infty,\alpha}\\
\lesssim& (1+\diam(\Omega)) \sum_{i=0}^{n}|(S_{\Omega}K)^{i}(Jg+S_{\Omega}\phi)|_{\infty,\alpha}+||Jg+S_{\Omega}\phi||_{\infty,\alpha}.
\end{split}
\end{equation*}
Furthermore, when $\delta > 0$ in \eqref{est_smallness_linear} is small enough, Corollary \ref{SK decay L infty estimate} yields
\begin{equation*}
|(S_{\Omega}K)^{i+2}(Jg+S_{\Omega}\phi)|_{\infty,\alpha}
\leq \frac{1}{2}|(S_{\Omega}K)^{i}(Jg+S_{\Omega}\phi) |_{\infty,\alpha}.
\end{equation*}
Hence, we get
\begin{align*}
&\sum_{i=0}^{2n-1}|(S_{\Omega}K)^{i}(Jg+S_{\Omega}\phi)|_{\infty,\alpha}\\
\leq& \sum_{i=1}^{n}|(S_{\Omega}K)^{2i-2}(Jg+S_{\Omega}\phi)|_{\infty,\alpha} +\sum_{i=1}^{n}|(S_{\Omega}K)^{2i-1}(Jg+S_{\Omega}\phi)|_{\infty,\alpha}\\
\leq& \sum_{i=1}^{n}\frac{1}{2^{i-1}}|Jg+S_{\Omega}\phi |_{\infty,\alpha} +\sum_{i=1}^{n}\frac{1}{2^{i-1}}|S_{\Omega}K(Jg+S_{\Omega}\phi) |_{\infty,\alpha}\\
\lesssim& |Jg+S_{\Omega}\phi |_{\infty,\alpha} +|S_{\Omega}K(Jg+S_{\Omega}\phi) |_{\infty,\alpha}\\
\lesssim& |Jg+S_{\Omega}\phi |_{\infty,\alpha}.
\end{align*}

In conclusion, the following estimate holds:
\begin{equation*}
\sum_{i=0}^{n}||(S_{\Omega}K)^i(Jg+S_{\Omega}\phi)||_{\infty,\alpha} \lesssim  (1+\diam(\Omega))|Jg+S_{\Omega}\phi|_{\infty, \alpha}+||Jg+S_{\Omega}\phi||_{\infty, \alpha},
\end{equation*}
which implies the convergence in $\hat{L}^{\infty}_{\alpha}$ of \eqref{Picard series}. This completes the proof of Lemma \ref{lem:existence linear}.

\section{Regularity for the nonlinear case}

To solve the nonlinear problem, we consider the following iteration scheme:
\begin{equation*}
\begin{cases}
v \cdot \nabla_{x}f_{1}+\nu(v)f_{1}=Kf_{1}, \\
v \cdot \nabla_{x}f_{i+1}+\nu(v)f_{i+1}=Kf_{i+1}+\Gamma(f_{i},f_{i}) \mbox{ for} \ i\geq 1, 
\end{cases}
(x, v) \in \Omega \times \R^3
\end{equation*}
with the boundary condition
\[
f_i(x,v)=g(x,v), \quad (x,v)\in \Gamma^-, i \geq 0.
\]
 
Our goal is to prove that the sequence of functions $\{ f_{i} \}$ converges in $\hat{L}^{\infty}_{\alpha}$ space. The key ingredients are to show the estimate in Lemma \ref{lem:bilinear_est} and to use this estimate in order to derive the convergence of the sequence $\{ f_i \}$.

To do this, we decompose the nonlinear term $\Gamma$ into two parts:
\[
\Gamma(h_1,h_2) = \pi^{-\frac{3}{4}} \left( \Gamma_{\gain}(h_1, h_2) - \Gamma_{\loss}(h_1, h_2) \right),
\]
where
\begin{align*}
\Gamma_{\gain}(h_1, h_2) :=& \int_{\mathbb{R}^{3}}\int_{0}^{2\pi}\int_{0}^{\frac{\pi}{2}}e^{-\frac{\vert v_{*} \vert^{2}}{2}}h_1(v')h_2(v'_{*}) B(\vert v-v_{*} \vert, \theta )\,d\theta \, d\phi\,dv_{*},\\
\Gamma_{\loss}(h_1, h_2) :=& \int_{\mathbb{R}^{3}}\int_{0}^{2\pi}\int_{0}^{\frac{\pi}{2}}e^{-\frac{\vert v_{*} \vert^{2}}{2}} h_1(v) h_2(v_{*}) B(\vert v-v_{*} \vert, \theta )\,d\theta \, d\phi\,dv_{*}
\end{align*}
We call $\Gamma_{\gain}$ and $\Gamma_{\loss}$ the gain term and the loss term, respectively.

We start from the weighted $L^\infty$ estimate for the nonlinear term.

\begin{lemma} \label{lem:Gamma}
Suppose \eqref{assumption_B1}. Then, for $h_1, h_2 \in L^{\infty}_{\alpha}$, we have
\[
|\Gamma(h_1, h_2)(x, v)| \lesssim |h_1|_{\infty, \alpha} |h_2|_{\infty, \alpha} e^{-\alpha |v|^2} (1 + |v|)^\gamma
\]
for a.e. $(x, v) \in \Omega \times \R^3$.
\end{lemma}

\begin{proof}
For the gain term, invoking the relation $|v|^2 + |v_*|^2 = |v'|^2 + |v_*'|^2$, we have
\begin{align*}
&\left| \Gamma_{\gain}(h_1,h_2)(x, v) \right|\\
\lesssim& \left|\int_{\mathbb{R}^{3}}\int_{0}^{2\pi}\int_{0}^{\frac{\pi}{2}}e^{-\frac{|v_{*}|^{2}}{2}}h_1(v')h_2(v'_{*})B(| v-v_{*}|, \theta )d \theta d\phi dv_{*}\right|\\
\lesssim& |h_1|_{\infty,\alpha}|h_2|_{\infty,\alpha}\int_{\mathbb{R}^{3}}\int_{0}^{2\pi}\int_{0}^{\frac{\pi}{2}}e^{-\frac{|v_{*}|^{2}}{2}}e^{-\alpha |v'|^{2}}e^{-\alpha |v_*'|^{2}}B(|v-v_{*}|, \theta )d \theta d\phi dv_{*}\\
=& |h_1|_{\infty,\alpha}|h_2|_{\infty,\alpha}e^{-\alpha |v|^{2}}\int_{\mathbb{R}^{3}}\int_{0}^{2\pi}\int_{0}^{\frac{\pi}{2}}e^{-\frac{|v_{*}|^{2}}{2}}e^{-\alpha |v_*|^{2}}B(|v-v_{*}|, \theta )d \theta d\phi dv_{*}\\
\lesssim& |h_1|_{\infty,\alpha}|h_2|_{\infty,\alpha}e^{-\alpha |v|^{2}}\int_{\mathbb{R}^{3}} e^{-\frac{|v_{*}|^{2}}{2}}e^{-\alpha |v_*|^{2}}|v-v_*|^{\gamma}\,dv_{*} \int_{0}^{\frac{\pi}{2}} \sin{\theta}\cos{\theta}\,d\theta\\
\lesssim& |h_1|_{\infty,\alpha}|h_2|_{\infty,\alpha}e^{-\alpha |v|^{2}}(1+|v|)^{\gamma}.
\end{align*}
Here, we used the estimate in Lemma \ref{lem:estimate_nonlin_1}.

In the same way, we obtain 
\[
\left|\Gamma_{\loss}(h_1,h_2)(x, v) \right| \lesssim |h_1|_{\infty,\alpha}|h_2|_{\infty,\alpha}e^{-\alpha |v|^{2}}(1+|v|)^{\gamma}.
\]

This completes the proof.
\end{proof}

As a corollary, we obtain the following estimate.

\begin{lemma} \label{lem:S Gamma}
Let $\Omega$ be a bounded domain in $\R^3$ and suppose \eqref{assumption_B1}. Then, for $h_1,h_2 \in L^{\infty}_{\alpha}$, we have
\begin{equation*}
|S_{\Omega} \Gamma(h_1,h_2)|_{\infty,\alpha} \lesssim (1+\diam(\Omega))|h_1|_{\infty,\alpha}|h_2|_{\infty,\alpha}.
\end{equation*}    
\end{lemma}

\begin{proof}
By Lemma \ref{lem:Gamma} with Lemma \ref{lem:estimate_nonlin_3} and the estimate \eqref{est:S infty}, we obtain
\begin{align*}
|S_{\Omega}\Gamma(h_1,h_2)(x,v)| \lesssim& |h_1|_{\infty,\alpha}|h_2|_{\infty,\alpha}e^{-\alpha|v|^2}(1+|v|)^{\gamma}\int_{0}^{\tau_{x,v}}e^{-\nu(v)s}\,ds\\
\lesssim& |h_1|_{\infty,\alpha}|h_2|_{\infty,\alpha}e^{-\alpha|v|^2}(1+|v|)^{\gamma} \min \left\{ 1,\frac{\diam(\Omega)}{|v|} \right\}\\
\lesssim& |h_1|_{\infty,\alpha}|h_2|_{\infty,\alpha}e^{-\alpha|v|^2}(1+\diam(\Omega))
\end{align*}
for a.e. $(x, v) \in \Omega \times \R^3$, which is the desired estimate. 
\end{proof}

We next give estimates for the $x$ derivatives of the nonlinear term.

\begin{lemma} \label{lem:Gamma_dx}
Suppose \eqref{assumption_B1}. Also, suppose {\bf Assumption $\Omega$} with uniform circumscribed and interior radii $R$ and $r$ respectively. Then, for $h_1,h_2 \in \hat{L}^{\infty}_{\alpha}$, we have
\begin{align*}
|\nabla_x \Gamma(h_1, h_2)(x, v)| \lesssim& \| h_1 \|_{\infty, \alpha} \| h_2 \|_{\infty, \alpha} e^{-\alpha |v|^2} \frac{R^{\frac{1}{2}}}{d_x^{\frac{1}{2}}} (1 + |v|)^\gamma\\
&+ \| h_1 \|_{\infty, \alpha} \| h_2 \|_{\infty, \alpha} w(x, v)^{-1} e^{-\alpha |v|^2} (1 + |v|)^\gamma
\end{align*}
for a.e. $(x, v) \in \Omega \times \R^3$.
\end{lemma}

\begin{proof}
We first treat the gain term. We notice that
\[
\nabla_{x}\Gamma_{\gain}(h_1,h_2)(x, v) = G_1(x, v) + G_2(x, v),
\]
where
\begin{equation*}
\begin{split}
G_1(x, v) :=& \int_{\mathbb{R}^{3}}\int_{0}^{2\pi}\int_{0}^{\frac{\pi}{2}}e^{-\frac{\vert v_{*} \vert^{2}}{2}}\nabla_{x}h_1(v')h_2(v'_{*})B(\vert v-v_{*} \vert, \theta )\,d\theta d\phi dv_{*},\\
G_2(x, v) :=& \int_{\mathbb{R}^{3}}\int_{0}^{2\pi}\int_{0}^{\frac{\pi}{2}}e^{-\frac{\vert v_{*} \vert^{2}}{2}}h_1(v')\nabla_{x}h_2(v'_{*})B(\vert v-v_{*} \vert, \theta)\,d\theta d\phi dv_{*}.
\end{split}
\end{equation*}

We recall that
\begin{align*}
|\nabla_{x} h_j (x, v)| \leq \| h_j \|_{\infty,\alpha} w(x, v)^{-1} e^{-\alpha|v|^{2}} 
\end{align*}
for $h_j \in \hat{L}^{\infty}_{\alpha}$, $j = 1, 2$. Thus, by Proposition \ref{prop:geometric estimate A}, we have
\begin{align*}
|G_1 (x, v)| \leq& \| h_1 \|_{\infty,\alpha} | h_2|_{\infty, \alpha}e^{-\alpha|v|^{2}}\\
& \times \int_{\mathbb{R}^{3}}\int_{0}^{2\pi}\int_{0}^{\frac{\pi}{2}}e^{-\frac{\vert v_{*} \vert^{2}}{2}}e^{-\alpha|v_{*}|^{2}} w(x, v')^{-1} B(\vert v-v_{*} \vert, \theta ) \,d\theta d\phi dv_{*}\\
\lesssim& \frac{R^\frac{1}{2}}{d_x^{\frac{1}{2}}} \| h_1 \|_{\infty,\alpha} \| h_2 \|_{\infty, \alpha} e^{-\alpha|v|^{2}}\\
& \times \int_{\mathbb{R}^{3}} \int_{0}^{2\pi} \int_{0}^{\frac{\pi}{2}}e^{-\frac{\vert v_{*} \vert^{2}}{2}}e^{-\alpha|v_{*}|^{2}} \left( 1 + \frac{1}{|v'|} \right) B(\vert v-v_{*} \vert, \theta ) \,d\theta d\phi dv_{*}.
\end{align*}
In the same way, we have
\begin{align*}
|G_2 (x, v)| \lesssim& \frac{R^\frac{1}{2}}{d_x^{\frac{1}{2}}} \| h_1 \|_{\infty,\alpha} \| h_2 \|_{\infty, \alpha} e^{-\alpha|v|^{2}}\\
& \times \int_{\mathbb{R}^{3}} \int_{0}^{2\pi} \int_{0}^{\frac{\pi}{2}}e^{-\frac{\vert v_{*} \vert^{2}}{2}}e^{-\alpha|v_{*}|^{2}} \left( 1 + \frac{1}{|v'|} \right) B(\vert v-v_{*} \vert, \theta ) \,d\theta d\phi dv_{*}.
\end{align*}

We give an estimate for the above integral factor. We notice that
\begin{align*}
&\int_{\mathbb{R}^{3}}\int_{0}^{2\pi}\int_{0}^{\frac{\pi}{2}}e^{-\frac{|v_{*}|}{2}}e^{-\alpha|v_{*}|^{2}} \left( 1+\frac{1}{|v'|} \right) B(|v-v_{*}|, \theta )\,d \theta d\phi dv_{*}\\
\lesssim& \int_{\mathbb{R}^{3}}\int_{0}^{2\pi}\int_{0}^{\frac{\pi}{2}}e^{-\frac{|v_{*}|^{2}}{2}}e^{-\alpha |v_{*}|^{2}} \left( 1+\frac{1}{|v'|} \right) |v-v_{*}|^{\gamma}\sin{\theta}\cos{\theta} \,d\theta d\phi dv_{*}\\
=& \int_{\mathbb{R}^{3}}\int_{\mathbb{S}^{2}}e^{-\frac{|v_{*}|^{2}}{2}}e^{-\alpha |v_{*}|^{2}} \left( 1+\frac{1}{|v'(\sigma)|} \right) |v-v_{*}|^{\gamma} \,d\Sigma(\sigma) dv_{*}\\
=& \int_{\mathbb{R}^{3}}e^{-\frac{|v_{*}|^{2}}{2}}e^{-\alpha |v_{*}|^{2}}|v-v_{*}|^{\gamma}\int_{\mathbb{S}^{2}} \left( 1+\frac{1}{|v'(\sigma)|} \right) \,d\Sigma(\sigma) dv_{*}.
\end{align*}
Here we introduced the sigma formulation, that is,
\[
    v':=\frac{v+v_*}{2}+\frac{|v_*-v|}{2}\sigma, \, v_{*}':=\frac{v+v_*}{2}-\frac{|v_*-v|}{2}\sigma, \quad \sigma \in \mathbb{S}^{2}.
\]
By taking unit vectors $e_2$, $e_3$ so that the pair $\{ \frac{v+v_*}{|v+v_*|} , e_2, e_3 \}$ forms an orthonormal basis in $\mathbb{R}^3$, and applying the changing of variable:
\[ 
\sigma=\cos{\psi} \frac{v_*+v}{|v_*+v|}+ (\sin{\psi}\cos{\phi}) e_2 + (\sin{\psi}\sin{\phi}) e_3, \quad 0 < \psi < \pi, 0 \leq \phi < 2\pi.
\]
we have
\begin{equation}\label{Estimate 1/v' S_2}
\begin{split}
&\int_{\mathbb{S}^{2}}\frac{1}{| v'(\sigma) |}\,d\Sigma(\sigma)\\
=& \int_{0}^{2\pi}\int_{0}^{\pi}\frac{1}{\sqrt{|\frac{v+v_{*}}{2} |^{2}+|\frac{v-v_{*}}{2}|^{2}+2|\frac{v+v_{*}}{2}|| \frac{v-v_{*}}{2}|\cos{\psi}}}\sin{\psi}\,d\psi d\phi\\
=& \int_{0}^{2\pi}\int_{-1}^{1}\frac{1}{\sqrt{|\frac{v+v_{*}}{2} |^{2}+|\frac{v-v_{*}}{2}|^{2}+2|\frac{v+v_{*}}{2}|| \frac{v-v_{*}}{2}|z}}\,dz d\phi\\
=& 2\pi \int_{-1}^{1}\frac{1}{\sqrt{|\frac{v+v_{*}}{2} |^{2}+|\frac{v-v_{*}}{2}|^{2}+2|\frac{v+v_{*}}{2}|| \frac{v-v_{*}}{2}|z}}\,dz\\
=& 8\pi \min \left\{ \frac{1}{|v+v_{*}|},\frac{1}{|v-v_{*}|} \right\}.
\end{split}
\end{equation}
By the identity \eqref{Estimate 1/v' S_2} with the aid of Lemma \ref{lem:estimate_nonlin_1} and Lemma \ref{lem:estimate_nonlin_2}, we obtain
\begin{equation*}
\begin{split}
&\int_{\mathbb{R}^{3}}e^{-\frac{|v_{*}|^{2}}{2}}e^{-\alpha |v_{*}|^{2}}|v-v_{*}|^{\gamma}\int_{\mathbb{S}^{2}} \left( 1+\frac{1}{|v'|} \right)\,d\Sigma(\sigma) dv_{*}\\
\lesssim &\int_{\mathbb{R}^{3}}e^{-\frac{|v_{*}|^{2}}{2}}e^{-\alpha |v_{*}|^{2}}|v-v_{*}|^{\gamma} \left( 1+\frac{1}{|v-v_{*}|} \right)\,dv_{*}\\
\lesssim &(1+|v|)^{\gamma}.
\end{split}
\end{equation*}

Hence, we have
\begin{align*}
\left| \nabla_{x}\Gamma_{\gain}(h_1,h_2)(x, v) \right| \leq& |G_1(x, v)| + |G_2(x, v)|\\ 
\lesssim& \| h_1 \|_{\infty, \alpha} \| h_2 \|_{\infty, \alpha} e^{-\alpha|v|^{2}} \frac{R^{\frac{1}{2}}}{d_{x}^{\frac{1}{2}}}(1+|v|)^{\gamma*}
\end{align*}
for a.e. $(x, v) \in \Omega \times \R^3$.

We next treat the loss term. We have
\begin{equation*}
\nabla_{x}\Gamma_{\loss}(h_1,h_2)(x, v) = L_1(x, v) + L_2(x, v),
\end{equation*}
where
\begin{align*}
L_1(x, v) :=& \nabla_x h_1(v)\int_{\mathbb{R}^{3}}\int_{0}^{2\pi}\int_{0}^{\frac{\pi}{2}}e^{-\frac{|v_{*}|^{2}}{2}}h_2(v_{*})B(|v-v_{*}|, \theta )d \theta d\phi dv_{*},\\
L_2(x, v) :=& h_1(v)\int_{\mathbb{R}^{3}}\int_{0}^{2\pi}\int_{0}^{\frac{\pi}{2}}e^{-\frac{|v_{*}|^{2}}{2}}\nabla_x h_2(v_{*})B(|v-v_{*}|, \theta )d \theta d\phi dv_{*}.
\end{align*}

For the $L_1$ term, invoking the assumption \eqref{assumption_B1} and Lemma \ref{lem:estimate_nonlin_1}, we get
\begin{align*}
|L_1(x, v)| \lesssim& |\nabla h_1 |_{\infty,\alpha, w} |h_2|_{\infty,\alpha} w(x, v)^{-1} e^{-\alpha|v|^{2}}\\
&\times \int_{\mathbb{R}^{3}}e^{-\frac{\vert v_{*} \vert^{2}}{2}} e^{-\alpha |v_*|^2}|v-v_*|^{\gamma}\,dv_{*} \int_{0}^{2\pi}\int_{0}^{\frac{\pi}{2}}\sin{\theta}\cos{\theta}\,d\theta d\phi\\
\lesssim& \| h_1 \|_{\infty,\alpha} |h_2|_{\infty,\alpha} w(x, v)^{-1} e^{-\alpha|v|^{2}}\int_{\mathbb{R}^{3}}e^{-\frac{\vert v_{*} \vert^{2}}{2}}e^{-\alpha |v_*|^2}|v-v_*|^{\gamma}\,dv_{*}\\
\lesssim& \| h_1 \|_{\infty,\alpha} \| h_2 \|_{\infty,\alpha} w(x, v)^{-1} e^{-\alpha|v|^{2}}(1+|v|)^{\gamma}
\end{align*}
for a.e. $(x, v) \in \Omega \times \R^3$. For the $L_2$ term, applying Proposition \ref{prop:geometric estimate C} and Lemma \ref{lem:estimate_nonlin_1}, we have
\begin{align*}
&|L_2(x, v)|\\ 
\lesssim& |h_1|_{\infty,\alpha} |\nabla_{x}h_2|_{\infty,\alpha,w} e^{-\alpha|v|^{2}}\\
&\times \int_{\mathbb{R}^{3}}\int_{0}^{2\pi}\int_{0}^{\frac{\pi}{2}} e^{-\frac{|v_{*}|^{2}}{2}} e^{-\alpha |v_*|^2} w(x, v_*)^{-1} B(|v-v_{*}|, \theta )\,d\theta d\phi dv_{*}\\
\lesssim& \| h_1 \|_{\infty,\alpha} \| h_2 \|_{\infty,\alpha} e^{-\alpha|v|^{2}}  \int_{\mathbb{R}^{3}}e^{-\frac{|v_{*}|^{2}}{2}}e^{-\alpha |v_*|^2} w(x, v_*)^{-1} |v-v_*|^{\gamma}\,dv_{*}\\
\lesssim& \| h_1 \|_{\infty,\alpha} \| h_2 \|_{\infty,\alpha} e^{-\alpha|v|^{2}} \frac{R^{\frac{1}{2}}}{d_{x}^{\frac{1}{2}}} \int_{\mathbb{R}^{3}} e^{-\frac{|v_{*}|^{2}}{2}}e^{-\alpha |v_*|^2} \ \left( 1 + \frac{1}{|v_*|} \right) |v-v_*|^{\gamma} \,dv_{*}\\
\lesssim& \| h_1 \|_{\infty,\alpha} \| h_2 \|_{\infty,\alpha} e^{-\alpha|v|^{2}}\frac{R^{\frac{1}{2}}}{d_{x}^{\frac{1}{2}}}(1+|v|)^{\gamma}
\end{align*}
for a.e. $(x, v) \in \Omega \times \R^3$. Thus, we obtain 
\begin{align*}
|\nabla_{x}\Gamma_{\loss}(h_1,h_2)(x,v)| \lesssim& \| h_1 \|_{\infty,\alpha} \| h_2 \|_{\infty,\alpha} w(x, v)^{-1} e^{-\alpha|v|^{2}}(1+|v|)^{\gamma}\\
&+ \| h_1 \|_{\infty,\alpha} \| h_2 \|_{\infty,\alpha} e^{-\alpha|v|^{2}}\frac{R^{\frac{1}{2}}}{d_{x}^{\frac{1}{2}}}(1+|v|)^{\gamma}
\end{align*}
for a.e. $(x, v) \in \Omega \times \R^3$.

The estimate for $|\nabla_x \Gamma(h_1, h_2)|$ is obtained by summing the estimates for $|\nabla_x \Gamma_{\gain}(h_1, h_2)|$ and $|\nabla_x \Gamma_{\loss}(h_1, h_2)|$.
\end{proof}

\begin{remark}
In the sigma formulation, $v'_*$ is obtained by replacing $\sigma$ in $v'$ by $-\sigma$. Thus, in the same way as in the computation \eqref{Estimate 1/v' S_2}, we have
\begin{equation}\label{Estimate 1/v_*' S_2}
\int_{\mathbb{S}^{2}}\frac{1}{| v_*'(\sigma) |}\,d\Sigma(\sigma) = 8\pi \min \left\{ \frac{1}{|v+v_{*}|},\frac{1}{|v-v_{*}|} \right\},
\end{equation}
which is used in a proof of Lemma \ref{lem:Gamma_dv}.
\end{remark}

\begin{lemma} \label{lem:S Gamma dx}
Suppose \eqref{assumption_B1}. Also, suppose {\bf Assumption $\Omega$} with uniform circumscribed and interior radii $R$ and $r$ respectively. Then, for $h_1,h_2 \in \hat{L}^{\infty}_{\alpha}$, we have
\[
|\nabla_x S_{\Omega} \Gamma (h_1,h_2)|_{\infty,\alpha,w} \lesssim \left( 1 + \diam(\Omega) + (Rr)^{\frac{1}{2}} \left(1+\frac{R}{r} \right) \right) \| h_1 \|_{\infty,\alpha} \| h_2 \|_{\infty,\alpha}.
\]
\end{lemma}

\begin{proof}
In the same way as in Section \ref{sec:linear}, we have
\begin{align*}
&|\nabla_{x}S_{\Omega} \Gamma(h_1,h_2)(x,v)|\\
\leq &\left|\frac{-n(q(x, v))}{|v|N(x,v)}\Gamma(h_1,h_2)(q(x,v),v)\right|+|S_{\Omega}\nabla_{x} \Gamma(h_1,h_2)(x,v)|.
\end{align*}
By Lemma \ref{lem:Gamma}, we have
\begin{align*}
\left|\frac{-n(q(x, v))}{|v|N(x,v)}\Gamma(h_1,h_2)(x,v)\right| \lesssim &\frac{1}{|v|N(x,v)}|h_1|_{\infty,\alpha}|h_2|_{\infty,\alpha}e^{-\alpha|v|^2}(1+|v|)^{\gamma}\\
\leq &w(x, v)^{-1} \| h_1 \|_{\infty,\alpha} \| h_2 \|_{\infty,\alpha}e^{-\alpha|v|^2}.
\end{align*}
Also, by Lemma \ref{lem:Gamma_dx}, we have
\begin{align*}
&|S_{\Omega}\nabla_{x}\Gamma (h_1,h_2)(x,v)|\\
\leq& \int_{0}^{\tau_{x,v}}e^{-\nu(v)s}\left|\nabla_{x}\Gamma (h_1,h_2)(x-sv,v)\right|\,ds\\
\lesssim& \| h_1 \|_{\infty, \alpha} \| h_2 \|_{\infty, \alpha} e^{-\alpha|v|^{2}} (1+|v|)^{\gamma} R^{\frac{1}{2}}\int_{0}^{\tau_{x,v}} \frac{1}{d_{x-vs}^{\frac{1}{2}}}\,ds\\
&+ \| h_1 \|_{\infty, \alpha} \| h_2 \|_{\infty, \alpha} e^{-\alpha|v|^{2}} (1+|v|)^{\gamma} \int_{0}^{\tau_{x,v}} e^{-\nu(v)s} w(x - sv, v)^{-1}\,ds.
\end{align*}

For the first term in the right hand side, we apply Corollary \ref{cor:geometric estimate B} to obtain
\begin{align*}
&\| h_1 \|_{\infty, \alpha} \| h_2 \|_{\infty, \alpha} e^{-\alpha|v|^{2}} (1+|v|)^{\gamma} R^{\frac{1}{2}}\int_{0}^{\tau_{x,v}} \frac{1}{d_{x-vs}^{\frac{1}{2}}}\,ds\\
\lesssim& \| h_1 \|_{\infty, \alpha} \| h_2 \|_{\infty, \alpha} e^{-\alpha|v|^{2}} \frac{1+|v|}{|v|} (Rr)^{\frac{1}{2}} \left( 1 + \frac{R}{r} \right)\\
\leq& \| h_1 \|_{\infty, \alpha} \| h_2 \|_{\infty, \alpha} e^{-\alpha|v|^{2}} w(x, v)^{-1} (Rr)^{\frac{1}{2}} \left( 1 + \frac{R}{r} \right).
\end{align*}
For the second term in the right hand side, noticing that $N(x - sv, v) = N(x, v)$ for all $0 < s < \tau_{x, v}$ and recalling Lemma \ref{lem:estimate_nonlin_3} and the estimate \eqref{est:S infty}, we have
\begin{align*}
&\| h_1 \|_{\infty, \alpha} \| h_2 \|_{\infty, \alpha} e^{-\alpha|v|^{2}} (1+|v|)^{\gamma} \int_{0}^{\tau_{x,v}} e^{-\nu(v)s} w(x - sv, v)^{-1}\,ds\\
\lesssim& \| h_1 \|_{\infty, \alpha} \| h_2 \|_{\infty, \alpha} w(x, v)^{-1} e^{-\alpha|v|^{2}} (1 + \diam(\Omega)).
\end{align*}

This completes the proof.
\end{proof}

Next we estimate the $v$ derivative of the nonlinear term.

\begin{lemma} \label{lem:Gamma_dv}
Suppose \eqref{assumption_B1}. Also, suppose {\bf Assumption $\Omega$} with uniform circumscribed and interior radii $R$ and $r$ respectively. Then, for $h_1,h_2 \in \hat{L}^{\infty}_{\alpha}$, we have
\begin{align*}
|\nabla_v \Gamma(h_1, h_2)(x, v)| \lesssim& \| h_1 \|_{\infty, \alpha} \| h_2 \|_{\infty, \alpha} e^{-\alpha |v|^2} \frac{R^{\frac{1}{2}}}{d_x^{\frac{1}{2}}} (1 + |v|)^\gamma\\
&+ \| h_1 \|_{\infty, \alpha} \| h_2 \|_{\infty, \alpha} w(x, v)^{-1} e^{-\alpha |v|^2} (1 + |v|)^\gamma\\
&+ \| h_1 \|_{\infty, \alpha} \| h_2 \|_{\infty, \alpha} e^{-\alpha |v|^2}
\end{align*}
for a.e. $(x, v) \in \Omega \times \R^3$.
\end{lemma}

\begin{proof}
For the gain term, by the sigma formulation, we have
\begin{align*}
\int_{0}^{2\pi}\int_{0}^{\frac{\pi}{2}}h_1(v')h_2(v'_{*})B(|v-v_{*}|, \theta)\,d \theta d\phi =& \int_{\mathbb{S}^2}h_1(v')h_2(v'_{*})\frac{B(|v-v_{*}|, \theta )}{2\sin{\theta}\cos{\theta}}\,d\Sigma(\sigma)\\
=& \frac{C}{2} \int_{\mathbb{S}^2}h_1(v')h_2(v'_{*})|v - v_*|^\gamma\,d\Sigma(\sigma).
\end{align*}
Here, we used the assumption \eqref{assumption_B1}. Thus, we obtain
\begin{align*}
|\nabla_v\Gamma_{\gain}(h_1,h_2)| \lesssim \int_{\mathbb{R}^{3}}e^{-\frac{|v_{*}|^{2}}{2}}\left|\nabla_v\int_{\mathbb{S}^2}h_1(v')h_2(v'_{*})|v-v_*|^{\gamma}\,d\Sigma(\sigma) \right|\,dv_{*}.
\end{align*}

Notice that, by the identity $|v|^2 + |v_*|^2 = |v'|^2 + |v_*'|^2$, we have
\begin{align*}
&\left|\nabla_v\int_{\mathbb{S}^2}h_1(v')h_2(v'_{*})|v-v_*|^{\gamma}\,d\Sigma(\sigma)\right|\\
\leq& \int_{\mathbb{S}^2}\left|\nabla_vh_1(v')\right||h_2(v'_{*})||v-v_*|^{\gamma}\,d\Sigma(\sigma) + \int_{\mathbb{S}^2}|h_1(v')|\left|\nabla_vh_2(v'_{*})\right||v-v_*|^{\gamma}\,d\Sigma(\sigma)\\
&+ \int_{\mathbb{S}^2}|h_1(v')| |h_2(v'_{*})||v-v_*|^{\gamma-1}\,d\Sigma(\sigma)\\
\lesssim& \| h_1 \|_{\infty,\alpha} |h_2|_{\infty,\alpha} |v-v_*|^{\gamma} \int_{\mathbb{S}^2} w(x, v')^{-1} e^{-\alpha(|v'|^{2} + |v_*'|^{2})} \,d\Sigma(\sigma) \\
&+ |h_1|_{\infty,\alpha} \| h_2 \|_{\infty,\alpha} |v-v_*|^{\gamma} \int_{\mathbb{S}^2} w(x, v_*')^{-1} e^{-\alpha(|v'|^{2} + |v_*'|^{2})} \,d\Sigma(\sigma)
\\&+ |h_1|_{\infty,\alpha} |h_2|_{\infty,\alpha} |v-v_*|^{\gamma-1} \int_{\mathbb{S}^2} e^{-\alpha(|v'|^{2} + |v_*'|^{2})} \,d\Sigma(\sigma)\\
\lesssim& \| h_1 \|_{\infty,\alpha} |h_2|_{\infty,\alpha} |v-v_*|^{\gamma} e^{-\alpha(|v|^{2} + |v_*|^{2})}\int_{\mathbb{S}^2} w(x, v')^{-1} \,d\Sigma(\sigma) \\
&+ |h_1|_{\infty,\alpha} \| h_2 \|_{\infty,\alpha} |v-v_*|^{\gamma} e^{-\alpha(|v|^{2} + |v_*|^{2})} \int_{\mathbb{S}^2} w(x, v_*')^{-1} \,d\Sigma(\sigma)
\\&+ |h_1|_{\infty,\alpha} |h_2|_{\infty,\alpha} |v-v_*|^{\gamma-1} e^{-\alpha(|v|^{2} + |v_*|^{2})}.
\end{align*}
For the first term, by Proposition \ref{prop:geometric estimate A} and the identity \eqref{Estimate 1/v' S_2}, we have
\begin{align*}
\int_{\mathbb{S}^2} w(x, v')^{-1} \,d\Sigma(\sigma) =& \int_{\mathbb{S}^2}\frac{1+|v'|}{|v'|N(x,v')}\,d\Sigma(\sigma) \\
\lesssim& \frac{R^{\frac{1}{2}}}{d_x^{\frac{1}{2}}} \int_{\mathbb{S}^2} \left(1 + \frac{1}{|v'|} \right)\,d\Sigma(\sigma)\\
\lesssim& \frac{R^{\frac{1}{2}}}{d_x^{\frac{1}{2}}}\left( 1+\frac{1}{|v-v_{*}|} \right).
\end{align*}
For the second term, by the identity \eqref{Estimate 1/v_*' S_2}, we have 
\begin{align*}
\int_{\mathbb{S}^2} w(x, v_*')^{-1} \,d\Sigma(\sigma)
\lesssim \frac{R^{\frac{1}{2}}}{d_x^{\frac{1}{2}}} \left( 1+\frac{1}{|v-v_{*}|} \right).
\end{align*}
Hence we have
\begin{align*}
&\left|\nabla_v \Gamma_{\gain}(h_1, h_2)(x, v) \right|\\
\lesssim& \| h_1 \|_{\infty,\alpha} \| h_2 \|_{\infty,\alpha}e^{-\alpha|v|^{2}} \frac{R^{\frac{1}{2}}}{d_x^{\frac{1}{2}}}\int_{\mathbb{R}^{3}}e^{-\frac{|v_{*}|^{2}}{2}}|v-v_*|^{\gamma}e^{-\alpha|v_*|^{2}} \left( 1+\frac{1}{|v-v_{*}|} \right) dv_{*}\\
&+ |h_1|_{\infty,\alpha}|h_2|_{\infty,\alpha}e^{-\alpha|v|^{2}}\int_{\mathbb{R}^{3}}e^{-\frac{|v_{*}|^{2}}{2}}e^{-\alpha|v_*|^{2}} |v-v_*|^{\gamma-1}dv_{*}\\
\lesssim&  \| h_1 \|_{\infty,\alpha} \| h_2 \|_{\infty,\alpha} e^{-\alpha|v|^{2}} \frac{R^{\frac{1}{2}}}{d_x^{\frac{1}{2}}} (1+|v|)^{\gamma} + \| h_1 \|_{\infty,\alpha} \| h_2 \|_{\infty,\alpha}e^{-\alpha|v|^{2}}
\end{align*}
for a.e. $(x, v) \in \Omega \times \R^3$.

For the loss term, we notice that 
\begin{align*}
\Gamma_{\loss}(h_1,h_2) = h_1(v)\int_{\mathbb{R}^{3}}e^{-\frac{|v_{*}|^{2}}{2}}h_2(v_{*}) \int_0^{2\pi} \int_0^{\frac{\pi}{2}} B(|v-v_{*}|, \theta )\,d\theta d\phi \,dv_{*},
\end{align*}
and, under the assumption \eqref{assumption_B1},
\[
\int_0^{2\pi} \int_0^{\frac{\pi}{2}} B(|v-v_{*}|, \theta )\,d\theta d\phi = C \pi |v - v_*|^\gamma.
\]
Thus, the gradient of $\Gamma_{\loss}$ with respect to the $v$ variable is described as below:
\begin{equation*}
\nabla_{v}\Gamma_{\loss}(h_1,h_2)(x,v) = L_3(x, v) + L_4(x, v),
\end{equation*}
where
\begin{align*}
L_3(x, v) :=& C \pi \nabla_{v}h_1(v) \int_{\mathbb{R}^{3}}e^{-\frac{|v_{*}|^{2}}{2}}h_2(v_{*}) |v - v_*|^\gamma \,dv_{*},\\
L_4(x, v) :=& C \gamma \pi h_1(v) \int_{\mathbb{R}^{3}}e^{-\frac{|v_{*}|^{2}}{2}}h_2(v_{*}) |v - v_*|^{\gamma - 2}(v - v_*) dv_{*}.
\end{align*}

For the $L_3$ term, we apply Lemma \ref{lem:estimate_nonlin_1} to obtain
\begin{align*}
|L_3(x, v)|\lesssim& |\nabla_v h_1|_{\infty,\alpha,w} |h_2|_{\infty,\alpha} w(x, v)^{-1} e^{-\alpha|v|^{2}}\int_{\mathbb{R}^{3}}e^{-\frac{\vert v_{*} \vert^{2}}{2}}e^{-\alpha|v_*|^2}|v-v_*|^{\gamma}dv_{*}\\
\lesssim& \| h_1 \|_{\infty,\alpha} |h_2|_{\infty,\alpha} w(x, v)^{-1} e^{-\alpha|v|^{2}}(1+|v|)^{\gamma}.
\end{align*}
Also, for the $L_4$ term, we apply Lemma \ref{lem:estimate_nonlin_2} to get
\begin{align*}
|L_4(x, v)| \lesssim& |h_1|_{\infty,\alpha} |h_2|_{\infty,\alpha} e^{-\alpha|v|^2} \int_{\mathbb{R}^{3}}e^{-\frac{|v_{*}|^{2}}{2}}e^{-\alpha|v_*|^2}| v-v_{*} |^{\gamma-1}\,dv_{*}\\
\lesssim& \| h_1 \|_{\infty,\alpha} \| h_2 \|_{\infty,\alpha} e^{-\alpha|v|^2}.
\end{align*}
Hence we conclude that
\begin{align*}
|\nabla_{v}\Gamma_{\loss}(h_1,h_2)(x,v)| \lesssim& \| h_1 \|_{\infty,\alpha} |h_2|_{\infty,\alpha} w(x, v)^{-1} e^{-\alpha|v|^{2}}(1+|v|)^{\gamma}\\ 
&+ \| h_1 \|_{\infty,\alpha} \| h_2 \|_{\infty,\alpha} e^{-\alpha|v|^2}
\end{align*}
for a.e. $(x, v) \in \Omega \times \R^3$.

Combining the estimates for $|\nabla_v \Gamma_{\gain}(h_1, h_2)|$ and $|\nabla_v \Gamma_{\loss}(h_1, h_2)|$, we obtain the desired estimate.
\end{proof}

\begin{lemma} \label{lem:S Gamma dv}
Suppose \eqref{assumption_B1}. Also, suppose {\bf Assumption $\Omega$} with uniform circumscribed and interior radii $R$ and $r$ respectively. Then, for $h_1,h_2 \in \hat{L}^{\infty}_{\alpha}$, we have
\begin{align*}
|\nabla_v S_{\Omega} \Gamma (h_1,h_2)|_{\infty,\alpha,w} \lesssim \left(1+\diam(\Omega) + (Rr)^{\frac{1}{2}} \left( 1+\frac{R}{r} \right) \right) \| h_1 \|_{\infty,\alpha} \| h_2 \|_{\infty,\alpha}.
\end{align*} 
\end{lemma}

\begin{proof}
We follow the proof of Lemma \ref{SK decay v gradient L infty estimate}. We have
\begin{align*}
\nabla_{v}S_{\Omega} \Gamma (h_1,h_2)(x,v) = S_{\Omega, v} \Gamma (h_1, h_2)(x, v) - (\nabla_v \nu(v)) S_{\Omega,s} \Gamma (h_1, h_2)(x, v)\\
-S_{\Omega, s} \nabla_x \Gamma (h_1, h_2)(x, v) + S_\Omega \nabla_v \Gamma (h_1, h_2)(x, v).
\end{align*}

For the first term, by Lemma \ref{lem:Gamma}, we have
\begin{align*} 
|S_{\Omega, v} \Gamma (h_1, h_2)(x, v)| \lesssim& \frac{\vert x-q(x,v) \vert  }{\vert v \vert^{2}N(x,v)}e^{-\nu(v)\tau_{x,v}}|h_1|_{\infty,\alpha}|h_2|_{\infty,\alpha}e^{-\alpha|v|^{2}}(1+|v|)^{\gamma}\\
\lesssim& \frac{1}{|v|N(x,v)} |h_1|_{\infty,\alpha} |h_2|_{\infty,\alpha} e^{-\alpha|v|^{2}}(1+|v|)^{\gamma}\\
\lesssim& w(x, v)^{-1} |h_1|_{\infty,\alpha} |h_2|_{\infty,\alpha} e^{-\alpha|v|^{2}}.
\end{align*}
For the second term, by the estimate \eqref{gradient nu decay estimate}, we get
\begin{align*} 
\left| (\nabla_v \nu(v)) S_{\Omega, s} \Gamma (h_1, h_2)(x, v) \right| \lesssim& |\Gamma(h_1, h_2)|_{\infty, \alpha} e^{-\alpha |v|^2} \int_0^{\tau_{x, v}} e^{-\frac{\nu(v)}{2}s}\,ds\\
\lesssim& |h_1|_{\infty, \alpha} |h_2|_{\infty, \alpha} e^{-\alpha |v|^2} \frac{\diam(\Omega)}{|v|}\\
\lesssim& \diam(\Omega) \| h_1 \|_{\infty, \alpha} \| h_2 \|_{\infty, \alpha} w(x, v)^{-1} e^{-\alpha |v|^2}.
\end{align*}

For the third term, we apply Lemma \ref{lem:Gamma_dx} to obtain
\begin{align*}
&\left| S_{\Omega, s} \nabla_{x}\Gamma (h_1,h_2)(x,v) \right|\\
\lesssim& \| h_1 \|_{\infty,\alpha} \| h_2 \|_{\infty,\alpha} e^{-\alpha|v|^{2}} R^{\frac{1}{2}}(1+|v|)^{\gamma} \int_{0}^{\tau_{x,v}} s e^{-\nu(v)s} \frac{1}{d_{x-sv}^{\frac{1}{2}}}\,ds\\
&+ \| h_1 \|_{\infty,\alpha} \| h_2 \|_{\infty,\alpha} e^{-\alpha|v|^{2}} (1+|v|)^{\gamma} \int_{0}^{\tau_{x,v}} s e^{-\nu(v)s} w(x - sv, v)^{-1}\,ds\\
\lesssim& \left( 1 + \diam(\Omega) + (Rr)^{\frac{1}{2}} \left( 1 + \frac{R}{r} \right) \right) \| h_1 \|_{\infty,\alpha} \| h_2 \|_{\infty,\alpha} w(x, v)^{-1} e^{-\alpha|v|^{2}}.
\end{align*}

For the last term, we use Lemma \ref{lem:Gamma_dv} to see that
\begin{align*}
&|S_\Omega \nabla_v \Gamma (h_1, h_2)(x, v)|\\ 
\lesssim& \| h_1 \|_{\infty,\alpha} \| h_2 \|_{\infty,\alpha} e^{-\alpha|v|^{2}} R^{\frac{1}{2}}(1+|v|)^{\gamma} \int_{0}^{\tau_{x,v}} s e^{-\nu(v)s} \frac{1}{d_{x-sv}^{\frac{1}{2}}}\,ds\\
&+ \| h_1 \|_{\infty,\alpha} \| h_2 \|_{\infty,\alpha} e^{-\alpha|v|^{2}} (1+|v|)^{\gamma} \int_{0}^{\tau_{x,v}} s e^{-\nu(v)s} w(x - sv, v)^{-1}\,ds\\
&+ \| h_1 \|_{\infty,\alpha} \| h_2 \|_{\infty,\alpha} e^{-\alpha|v|^{2}} \int_{0}^{\tau_{x,v}} s e^{-\nu(v)s} w(x - sv, v)^{-1}\,ds\\
\lesssim& \left( 1 + \diam(\Omega) + (Rr)^{\frac{1}{2}} \left( 1 + \frac{R}{r} \right) \right) \| h_1 \|_{\infty,\alpha} \| h_2 \|_{\infty,\alpha} w(x, v)^{-1} e^{-\alpha|v|^{2}}.
\end{align*}

This completes the proof.
\end{proof}

Lemma \ref{lem:bilinear_est} follows from Lemma \ref{lem:S Gamma}, Lemma \ref{lem:S Gamma dx} and Lemma \ref{lem:S Gamma dv}.

We are ready to prove Theorem \ref{main theorem nonlinear}. First, define $f_0:=0$ and consider the following iteration scheme:
\begin{equation*} 
\begin{cases}
 &v \cdot \nabla_{x}f_{i+1} + \nu(v)f_{i+1} = Kf_{i+1} + \Gamma(f_{i},f_{i}), \quad (x, v) \in \Omega \times \R^3,\\
 &f_i(x,v)=g(x,v), \quad (x,v) \in \Gamma^-
 \end{cases}
\end{equation*}
for $i\geq 1$. By Lemma \ref{lem:existence linear} and Lemma \ref{lem:bilinear_est}, we have 
\begin{align*}
\| f_{i+1} \|_{\infty,\alpha} \lesssim& \| S_{\Omega}\Gamma(f_{i},f_{i})| \_{\infty,\alpha} + \| Jg \|_{\infty,\alpha}\\
\lesssim& \left(1+\diam(\Omega)+(Rr)^{\frac{1}{2}}\left(1+\frac{R}{r}\right)\right) \| f_i \|_{\infty,\alpha} \| f_i \|_{\infty,\alpha} + \| Jg \|_{\infty,\alpha}.
\end{align*}
Hence, by the assumption that $\diam(\Omega)$ and $(Rr)^{\frac{1}{2}}\left(1+\frac{R}{r}\right)$ is small enough, we have 
\[
||f_{i+1}||_{\infty,\alpha} 
\leq C ||f_i||_{\infty,\alpha} ||f_i||_{\infty,\alpha} +C||Jg||_{\infty,\alpha}.
\]
for some constant $C>0$. Without loss of generality, we assume that $C$ is a constant any arbitrary large real number which is greater than $1$.

We further take $\delta > 0$ so small that$\delta < 1/4C^2$ to achieve 
\[
\| f_1 \|_{\infty, \alpha} \leq \frac{1}{4C} \leq \frac{1}{2C}.
\] 
Also, if $\| f_i \|_{\infty, \alpha} \leq 1/2C$ for some $i$, we have 
\begin{align*}
||f_{i+1}||_{\infty,\alpha} 
\leq &\frac{1}{2} ||f_i||_{\infty,\alpha} +\frac{1}{4C} \leq \frac{1}{2C}.
\end{align*}
Hence, by induction, the $\hat{L}^\infty_\alpha$ norm of the sequence $f_i$ is uniformly bounded by $1/2C$. Furthermore, by substituting $f_{i+1}-f_{i}$ for \eqref{Nonlinear Boltmann iteration scheme}, we have 
\begin{equation*}
\begin{cases}
 &v \cdot \nabla_{x}(f_{i+1}-f_{i}) + \nu(v)(f_{i+1}-f_{i})\\
 &\quad = K(f_{i+1}-f_{i})+\Gamma(f_{i},f_{i})-\Gamma(f_{i-1},f_{i-1}), \quad (x, v) \in \Omega \times \R^3,\\
 &f_{i+1}(x,v)-f_i(x,v)=0, (x,v) \in \Gamma^-.
 \end{cases}
 \end{equation*}
Notice that $\Gamma(f_{i},f_{i})-\Gamma(f_{i-1},f_{i-1})=\Gamma(f_{i},f_{i}-f_{i-1})+\Gamma(f_{i}-f_{i-1},f_{i-1})$. Hence, we have
\begin{align*}
\| f_{i+1}-f_{i} \|_{\infty,\alpha} \lesssim& \left( 1 + \diam(\Omega)+(Rr)^{\frac{1}{2}} \left( 1+\frac{R}{r} \right) \right)\\
& \times ( \| f_i \|_{\infty,\alpha} \| f_{i}-f_{i-1} \|_{\infty,\alpha} + \| f_{i}-f_{i-1} \|_{\infty,\alpha} \| f_{i-1} \|_{\infty,\alpha})\\ \lesssim&
\left( 1 + \diam(\Omega)+(Rr)^{\frac{1}{2}} \left( 1+\frac{R}{r} \right) \right) \frac{1}{C} \| f_{i}-f_{i-1} \|_{\infty,\alpha}.
\end{align*}
In the last line, we use the uniform bound $1/2C$ of $\| f_{i+1} \|_{\infty,\alpha}$.
With small enough $\diam(\Omega)$ and $(Rr)^{\frac{1}{2}} \left( 1+\frac{R}{r} \right)$ we have 
\begin{equation*} 
\| f_{i+1}-f_{i} \|_{\infty,\alpha} \lesssim \frac{1}{C} \| f_{i}-f_{i-1} \|_{\infty,\alpha}.
\end{equation*}
With large $C$, we finally deduced that
\begin{equation*} 
\| f_{i+1}-f_{i} \|_{\infty,\alpha} \leq\frac{1}{2} \| f_{i}-f_{i-1} \|_{\infty,\alpha}.
\end{equation*}

Hence we achieve the convergence in $\hat{L}^{\infty}_{\alpha}$ of the iteration scheme \eqref{Nonlinear Boltmann iteration scheme} when $\diam(\Omega)$, $(Rr)^{\frac{1}{2}} \left( 1+\frac{R}{r} \right)$ and $\| Jg \|_{\infty,\alpha}$ is small enough. 

\section*{Acknowledgement}

 I. Chen is supported in part by NSTC with the grant number 108-2628-M-002-006-MY4, 111-2918-I-002-002-, and 112-2115-M-002-009-MY3. C. Hsia is supported in part by NSTC with the grant number 109-2115-M-002-013-MY3. D. Kawagoe is supported in part by JSPS KAKENHI grant number 20K14344. D. Kawagoe and J. Su is supported in part by the National Taiwan University-Kyoto University Joint Funding.

\appendix

\section{The inclusion $\hat{L}^{\infty}_{\alpha} \subset W^{1, p}$}

In this section, we provide a detail proof of the fact that the weighted space $\hat{L}^{\infty}_{\alpha}$ is in fact in $W^{1,p}$ for $1 \leq p < 3$.

\begin{proposition} \label{appendix 1}
Suppose that the domain $\Omega$ satisfies the uniform circumscribed condition with radius $R$. Then, given $0< \alpha < \infty$ we have $\hat{L}^{\infty}_{\alpha}\subseteq W^{1,p}$ for $1\leq p < 3$.
\end{proposition}

\begin{proof}
Given a function $h \in \hat{L}^{\infty}_{\alpha}$, it suffices to show that $h$, $\nabla_x h$ and $\nabla_v h$ belong to $L^p$. Without loss of generality, it suffices to show that $w(x, v)^{-1}$ belongs to $L^p$.

By performing the change of variable $x=z+sv/|v|$, where $z \in \Gamma^-_v$, we have
\begin{align*}
&\int_{\mathbb{R}^3}\int_{\Omega} w(x, v)^{-p} e^{-p\alpha|v|^2}\,dxdv\\
\lesssim& \int_{\mathbb{R}^3}\int_{\Gamma^-_v}\int_{0}^{q(x,-v)} \left( 1+\frac{1}{|v|^p} \right) \frac{1}{N(x,v)^{p-1}}e^{-p\alpha|v|^2}\,dsd\Sigma(z) dv.
\end{align*}
By Proposition \ref{prop:geometric estimate C}, we have
\begin{align*}
&\int_{\mathbb{R}^3}\int_{\Gamma^-_v}\int_{0}^{q(x,-v)} \left( 1+\frac{1}{|v|^p} \right) \frac{1}{N(x,v)^{p-1}}e^{-p\alpha|v|^2}\,dsd\Sigma(z) dv\\
\lesssim& \int_{\mathbb{R}^3}\int_{\Gamma^-_v} \left( 1+\frac{1}{|v|^p} \right) \frac{1}{N(x,v)^{p-2}}e^{-p\alpha|v|^2}\,d\Sigma(z) dv.
\end{align*}
We then change the order of integration and introduce the spherical coordinates to obtain
\begin{align*}
&\int_{\mathbb{R}^3}\int_{\Gamma^-_v} \left( 1+\frac{1}{|v|^p} \right)\frac{1}{N(x,v)^{p-2}}e^{-p\alpha|v|^2}\,d\Sigma(z) dv\\
=& \int_{\partial \Omega}\int_{\Gamma^-_x} \left( 1+\frac{1}{|v|^p} \right)\frac{1}{N(x,v)^{p-2}}e^{-p\alpha|v|^2}\,dvd\Sigma(z)\\
=& \int_{\partial \Omega}\int_{0}^{\infty}\int_0^{2 \pi} \int_0^{\frac{\pi}{2}} \left( 1+\frac{1}{r^p} \right) \frac{1}{\cos^{p-2}{\theta}}e^{-p\alpha r^2}r^2 \sin{\theta} \,d\theta d\phi drd\Sigma(z)\\
\lesssim& \int_{\partial \Omega}\int_{0}^{\infty}\int_0^{2 \pi} \int_0^{\frac{\pi}{2}} \left( r^2+\frac{1}{r^{p-2}} \right) \frac{1}{\cos^{p-2}{\theta}}e^{-p\alpha r^2}\sin{\theta} \,d\theta d\phi drd\Sigma(z)\\
\lesssim& \int_{\partial \Omega}\int_{0}^{\infty}\int_0^{\frac{\pi}{2}} \left( r^2+\frac{1}{r^{p-2}} \right) \frac{1}{\cos^{p-2}{\theta}}e^{-p\alpha r^2}\sin{\theta} \,d\theta  drd\Sigma(z)\\
=& \int_{\partial \Omega}\int_{0}^{\infty} \left( r^2+\frac{1}{r^{p-2}} \right) e^{-p\alpha r^2}\int_0^{1} \frac{1}{t^{p-2}}\, dt drd\Sigma(z)\\
<& \infty.
\end{align*}
This completes the proof.
\end{proof}

\end{document}